\documentclass[12pt,oneside,a4papre]{amsart}

\usepackage{amssymb}
\usepackage{amsmath}
\usepackage{amsthm}
\usepackage{amscd}

\usepackage{longtable}
\usepackage{mathrsfs}

\usepackage[dvipdfmx]{xcolor}
\usepackage[dvipdfmx]{pict2e}
\usepackage[dvipdfmx]{graphicx}
\usepackage{comment}
\usepackage{todonotes}

\usepackage{tikz}

\usepackage[all]{xy}

\usetikzlibrary{cd}

\theoremstyle{plain}
\newtheorem{theorem}{Theorem}[section]
\newtheorem{lemma}[theorem]{Lemma}
\newtheorem{corollary}[theorem]{Corollary}
\newtheorem{proposition}[theorem]{Proposition}
\newtheorem{remark}[theorem]{Remark}

\theoremstyle{definition}
\newtheorem{definition}[theorem]{Definition}

\newcommand{\kah}{K\"{a}hler }

\newcommand{\idd}{i\partial\overline{\partial}}


\begin{document}
\title
[Bogomolov-Sommese type vanishing for holomorphic vector bundles]
{Bogomolov-Sommese type vanishing theorem for holomorphic vector bundles equipped with positive singular Hermitian metrics}
\author{Yuta Watanabe}
\date{}

\begin{abstract}
   In this article, we obtain the Bogomolov-Sommese type vanishing theorem involving multiplier ideal sheaves for big line bundles.
   We define a dual Nakano semi-positivity of singular Hermitian metrics with $L^2$-estimates
   and prove a vanishing theorem which is a generalization of the Bogomolov-Sommese type vanishing theorem to holomorphic vector bundles. 
\end{abstract}

\vspace{-5mm}

\maketitle

\tableofcontents

\vspace{-8mm}

\section{Introduction}\label{section:1}

Positivity notions for holomorphic vector bundles and multiplier ideal sheaves 
play an important role in several complex variables and complex algebraic geometry. 
For holomorphic vector bundles, singular Hermitian metrics and its positivity are very interesting subjects.
On holomorphic line bundles, positivity of a singular Hermitian metric corresponds to plurisubharmonicity of the local weight and the multiplier ideal sheaf is an invariant of the singularities of the plurisubharmonic functions.

Let $X$ be a complex manifold and $\varphi$ be a plurisubharmonic function. Let $\mathscr{I}(\varphi)$ be the sheaf of germs of holomorphic functions $f$ such that $|f|^2e^{-\varphi}$ is locally integrable which is called the multiplier ideal sheaf.
Let $h$ be a singular Hermitian metric on a holomorphic line bundle $L$ over $X$ and $\varphi$ be the local weight of $h$, i.e. $h=e^{-\varphi}$. Then we define the multiplier ideal sheaf for $h$ by $\mathscr{I}(h):=\mathscr{I}(\varphi)$.

For a holomorphic line bundle $L$ over a projective manifold $X$ of $\mathrm{dim}\,X=n$,
the famous Bogomolov-Sommese vanishing theorem \cite{Bog78} asserts that $H^0(X,\Omega_X^p\otimes L^*)=0$ for $p<\kappa(L)$.
In particular, if $L$ is big then we have that 
\begin{align*}
    H^n(X,\Omega_X^p\otimes L)=0
\end{align*}
for $p>0$ by taking the dual.
The Bogomolov-Sommese type vanishing theorem have been studied in the direction of weakening the positivity (cf.\,\cite{Mou98},\,\cite{Wu20}). 

For big line bundles, we first obtain the following Bogomolov-Sommese type vanishing theorem which involves a multiplier ideal sheaf as in the Demailly-Nadel vanishing theorem (cf.\,\cite{Nad89}, \cite{Dem93})
and which is an extension of the Demailly-Nadel vanishing theorem to $(p,n)$-forms.

\begin{theorem}\label{Ext Nadel V-thm}
    Let $X$ be a projective manifold of dimension $n$ equipped with a \kah metric $\omega$ on $X$. 
    Let $L$ be a holomorphic line bundle on $X$ equipped with a singular Hermitian metric $h$.
    We assume that 
    \begin{align*}
        i\Theta_{L,h}\geq\varepsilon\omega
    \end{align*}
    in the sense of currents for some $\varepsilon>0$. Then we have that
    \begin{align*}
        H^n(X,\Omega_X^p\otimes L\otimes\mathscr{I}(h))&=0
    \end{align*}
    for $p>0$. 
\end{theorem}

Theorem \ref{Ext Nadel V-thm} is shown using the $L^2$-estimate theorem (see Theorem \ref{line bdl big L2-estimate for (p,n)-forms}) for $(p,n)$-forms and a fine resolution of $\Omega_X^p\otimes L\otimes \mathscr{I}(h)$.

Notions of singular Hermitian metrics for holomorphic vector bundles were introduced and investigated (cf.\,\cite{BP08}, \cite{deC98}). 
However, it is known that we cannot always define the curvature currents with measure coefficients (see \cite{Rau15}). 
Hence, Griffiths semi-negativity or semi-positivity (\cite{BP08},\,\cite{Rau15}, see Definition \ref{sing Griffiths def}) and Nakano semi-negativity (\cite{Rau15}, see Definition \ref{def Nakano semi-negative as Raufi}) is defined without using the curvature currents by using the properties of plurisubharmonic functions.
Here, Griffiths semi-positivity can be returned to Griffiths semi-negativity by considering the duality, but this method cannot be used for Nakano semi-positivity because the dual of a Nakano negative bundle in general is not Nakano positive.

After that, Nakano semi-positivity for singular Hermitian metrics (see Definition \ref{def Nakano semi-posi sing}) was defined in \cite{Ina21b}, which establishes the singular-type Nakano vanishing theorem, i.e. the Demailly-Nadel type vanishing theorem extended to holomorphic vector bundles.
This definition is based on characterizations of positivity using $L^2$-estimates for $(n,1)$-forms (cf.\,\cite{DWZZ18}, \cite{DNWZ20}, \cite{HI20}) and does not require the use of curvature currents.
In \cite{Wat21}, these characterizations of positivity using $L^2$-estimates for $(n,1)$-forms are extended to $(n,q)$ and $(p,n)$-forms. 

Throughout this paper, we let $X$ be an $n$-dimensional complex manifold and $E\to X$ be a holomorphic vector bundle of finite rank $r$.
From the definition of Nakano semi-negativity (\cite{Rau15}, see Definition \ref{def Nakano semi-negative as Raufi}), we naturally define dual Nakano semi-positive singular Hermitian metrics (see Definition \ref{def dual Nakano semi-posi sing}) with characterization using $L^2$-estimates (see Proposition \ref{dual Nakano semi-posi sing then has L2-estimate condition}). 
Then, by using the proof method of Theorem \ref{Ext Nadel V-thm}, we obtain the following vanishing theorem 
which is a generalization of dual Nakano vanishing theorem to singular Hermitian metrics and of the Bogomolov-Sommese vanishing theorem to holomorphic vector bundles.

\begin{theorem}\label{sing dual Nakano V-thm}
    Let $X$ be a projective manifold equipped with a \kah metric $\omega_X$ on $X$.
    We assume that $(E,h)$ is strictly dual Nakano $\delta_{\omega_X}$-positive in the sense of Definition \ref{sing strictly positive of dual Nakano} on $X$ and $\mathrm{det}\,h$ is bounded on $X$.
    Then we have the following vanishing:
    \begin{align*}
        H^n(X,\Omega_X^p\otimes E)=0,
    \end{align*}
    for $p>0$.
\end{theorem}

We get the following result which is a generalization of the Griffiths vanishing theorem (cf.\,[Dem-book,\,ChapterVII,\,Corollary\,9.4], \cite{LSY13}) to singular Hermitian metrics 
and which can also be considered as a generalization of the Demailly-Nadel vanishing theorem and Theorem \ref{Ext Nadel V-thm} to holomorphic vector bundles.
Here, the generalization up to $(n,q)$-forms for singular Hermitian metrics is already known in \cite{Ina21b}.

\begin{theorem}\label{Ext sing Griffiths V-thm}
    Let $X$ be a projective manifold equipped with a \kah metric $\omega_X$ on $X$. We assume that $(E,h)$ is strictly Griffiths $\delta_{\omega_X}$-positive in the sense of Definition \ref{sing strictly positive of Grif and Nakano} on $X$.
    Then we have the following vanishing:
    \begin{align*}
        H^q(X,K_X\otimes\mathscr{E}(h\otimes\mathrm{det}\,h))&=0\\
        H^n(X,\Omega_X^p\otimes\mathscr{E}(h\otimes\mathrm{det}\,h))&=0,
    \end{align*}
    for $p,q>0$.
\end{theorem}

\section{Proof of Theorem \ref{Ext Nadel V-thm}}

In this section, we first prove Theorem \ref{line bdl big L2-estimate for (p,n)-forms} and then use it to show Theorem \ref{Ext Nadel V-thm}.

\begin{theorem}\label{line bdl big L2-estimate for (p,n)-forms}
    Let $X$ be a projective manifold of dimension $n$ and $\omega$ be a \kah metric on $X$. Let $L$ be a holomorphic line bundle equipped with a singular Hermitian metric $h$ whose local weights are denoted $\varphi\in L^1_{loc}$, i.e. $h=e^{-\varphi}$.
    We assume that 
    \begin{align*}
        i\Theta_{L,h}=\idd\varphi\geq\varepsilon\omega
    \end{align*}
    in the sense of currents for some $\varepsilon>0$. 
    Then for any $f\in L^2_{p,n}(X,L,h,\omega)$ satisfying $\overline{\partial}f=0$, there exists $u\in L^2_{p,n-1}(X,L,h,\omega)$ such that $\overline{\partial}u=f$ and 
    \begin{align*}
        \int_X|u|^2_{h,\omega}dV_\omega\leq\frac{1}{p\varepsilon}\int_X|f|^2_{h,\omega}dV_\omega.
    \end{align*}
\end{theorem}

First, we consider Theorem \ref{line bdl big L2-estimate for (p,n)-forms} on a Stein manifold (= Proposition \ref{L2-estimate for (p,n)-forms on Stein}) and consider Lemma \ref{uniformity of (p,n)-forms} to show this.
Here, the claim of the type of Theorem \ref{line bdl big L2-estimate for (p,n)-forms} and Lemma \ref{uniformity of (p,n)-forms} for $(n,q)$-forms rather than $(p,n)$-forms is already known (see [Dem-book, ChapterVIII], \cite{Dem93}).

Let $(X,\omega)$ be a Hermitian manifold and $(E,h)$ be a holomorphic Hermitian vector bundle over $X$.
We denote the curvature operator $[i\Theta_{E,h},\Lambda_\omega]$ on $\Lambda^{p,q}T^*_X\otimes E$ by $A^{p,q}_{E,h,\omega}$.
And the fact that the curvature operator $[i\Theta_{E,h},\Lambda_\omega]$ is positive (resp. semi-positive) definite on $\Lambda^{p,q}T^*_X\otimes E$ is simply written as $A^{p,q}_{E,h,\omega}>0$ (resp. $\geq 0$).

\begin{lemma}\label{uniformity of (p,n)-forms}
    Let $(E,h)$ be a holomorphic Hermitian vector bundle over $X$ and $\omega, \gamma$ be Hermitian metrics on $X$ such that $\gamma\geq\omega$. 
    For any $u\in\Lambda^{p,n}T^*_X\otimes E, p\geq1$, we have that $|u|^2_{h,\gamma}dV_\gamma\leq|u|^2_{h,\omega}dV_\omega$ and that if $A^{p,n}_{E,h,\omega}>0$ $($resp. $\geq0)$ then
    \begin{align*}
        A^{p,n}_{E,h,\gamma}>0~(\textit{resp.}~\geq0), \quad \langle(A^{p,n}_{E,h,\gamma})^{-1}u,u\rangle_{h,\gamma} dV_\gamma\leq\langle(A^{p,n}_{E,h,\omega})^{-1}u,u\rangle_{h,\omega} dV_\omega.
    \end{align*}
\end{lemma}

To show Lemma \ref{uniformity of (p,n)-forms}, we use the following symbolic definition and lemma which is the calculation results.

\begin{definition}$\mathrm{(cf.\,[Wat22,\,Definition\,2.1])}$
    Let $(M,g)$ be an oriented Riemmannian $\mathcal{C}^\infty$-manifold with $\mathrm{dim}_{\mathbb{R}}M=m$ and $(\xi_1,\cdots,\xi_m)$ be an orthonormal basis of $(T_M,g)$ at $x_0\in M$.
    For any ordered multi-index $I$, we define $\varepsilon(s,I)\in \{-1,0,1\}$ as the number that satisfies $\xi_s\lrcorner~\xi_I^*=\varepsilon(s,I)\xi^*_{I\setminus s}$, 
    where if $s\notin I$ then $\varepsilon(s,I)=0$ and if $s\in I$ then $\varepsilon(s,I)\in\{-1,1\}$. Here, the symbol $\bullet\lrcorner~\bullet$ represents the interior product, i.e. $\xi_s\lrcorner~\xi_I^*=\iota_{\xi_s}\xi_I^*$.

    Let $(X,\omega)$ be a Hermitian manifold of $\dim_\mathbb{C}X=n$. If $(\partial/\partial z_1,\cdots,\partial/\partial z_n)$ is an orthonormal basis of $(T_X,\omega)$ at $x_0$ then we define $\varepsilon(s,I)$ in the same way as follows 
    $\displaystyle \frac{\partial}{\partial z_s}\lrcorner~ dz_I=\varepsilon(s,I)dz_{I\setminus s}$. In particular, we have that $\displaystyle \frac{\partial}{\partial \overline{z}_s}\lrcorner~ d\overline{z}_I=\varepsilon(s,I)d\overline{z}_{I\setminus s}$.
\end{definition}

\begin{lemma}\label{calculate (p,q)-forms}$\mathrm{(cf.\,[Wat22,\,Proposition\,2.2])}$
    Let $(X,\omega)$ be a Hermitian manifold and $(E,h)$ be a holomorphic Hermitian vector bundle over $X$.
    Let $x_0\in X$ and $(z_1,\ldots,z_n)$ be local coordinates such that $(\partial/\partial z_1,\ldots,\partial/\partial z_n)$ is an orthonormal basis of $(T_X,\omega)$ at $x_0$. Let $(e_1,\ldots,e_r)$ be an orthonormal basis of $E_{x_0}$. We can write
\begin{align*}
    \omega_{x_0}=i\sum_{1\leq j\leq n}dz_j\wedge d\overline{z}_j,\quad
    i\Theta_{E,h,x_0}=i\sum_{j,k,\lambda,\mu}c_{jk\lambda\mu}dz_j\wedge d\overline{z}_k\otimes e^*_\lambda\otimes e_\mu.
\end{align*}
Let $J,K,L$ and $M$ be ordered multi-indices with $|J|=|L|=p$ and $|K|=|M|=q$.
For any $(p,q)$-form $ u=\sum_{|J|=p,|K|=q,\lambda}u_{J,K,\lambda}dz_J\wedge d\overline{z}_K\otimes e_\lambda\in\Lambda^{p,q}T^*_{X,x_0}\otimes E_{x_0}$, we have the following calculation results:
\begin{align*}
    \langle[i\Theta_{E,h},\Lambda_\omega]u,u\rangle_\omega&=\Bigl(\sum_{j\in J}+\sum_{j\in K}-\sum_{1\leq j\leq n}\Bigr)c_{jj\lambda\mu}u_{J,K,\lambda}\overline{u}_{J,K,\mu}\\
    &+\sum_{j\ne k,K\setminus j=M\setminus k}c_{jk\lambda\mu}u_{J,K,\lambda}\overline{u}_{J,M,\mu}\varepsilon(j,K)\varepsilon(k,M)\\
    &+\sum_{j\ne k,L\setminus j=J\setminus k}c_{jk\lambda\mu}u_{L,K,\lambda}\overline{u}_{J,K,\mu}\varepsilon(k,J)\varepsilon(j,L).
\end{align*}
\end{lemma}

$\textit{Proof of Lemma \ref{uniformity of (p,n)-forms}}$.
    For any $x_0\in X$, after a linearly transformation, we may assume $\omega=i\sum^n_{j=1}dz_j\wedge d\overline{z}_j$ and $\gamma=i\sum^n_{j=1}\gamma^2_jdz_j\wedge d\overline{z}_j$ at $x_0$ with $\gamma_j\geq1$.
    Let $w_j=\gamma_jz_j$ for $j=1,2,\cdots,n$ and $(e_1,\ldots,e_r)$ be an orthonormal basis of $E_{x_0}$. Then we can write
    \begin{align*}
        \gamma&=i\sum_{1\leq j\leq n}dw_j\wedge d\overline{w}_j,\\
        i\Theta_{E,h}&=i\sum_{j,k,\lambda,\mu}c_{jk\lambda\mu}dz_j\wedge d\overline{z}_k\otimes e^*_\lambda\otimes e_\mu=i\sum_{j,k,\lambda,\mu}c'_{jk\lambda\mu}dw_j\wedge d\overline{w}_k\otimes e^*_\lambda\otimes e_\mu
    \end{align*}
    with $c'_{jk\lambda\mu}=c_{jk\lambda\mu}/\gamma_j\gamma_k$. For any ordered multi-indices $J$ we denote $\gamma_J=\Pi_{j\in J}\gamma_j$. For any $u\in\Lambda^{p,n}T^*_{X,x_0}\otimes E_{x_0}$ we can write
    \begin{align*}
        u=\sum u_{J\lambda}dz_J\wedge d\overline{z}_N\otimes e_\lambda=\sum u'_{J\lambda}dw_J\wedge d\overline{w}_N\otimes e_\lambda
    \end{align*}
    with $u'_{J\lambda}=u_{J\lambda}/\gamma_J\gamma_N$ where $N=\{1,\cdots,n\}$.

    Then we obtain that 
    \begin{align*}
        &|u|^2_{h,\gamma}=\sum|u'_{J\lambda}|^2=\sum\gamma^{-2}_J\gamma^{-2}_N|u_{J\lambda}|^2, \qquad dV_\gamma=\gamma_N^2dV_\omega,\\
        &|u|^2_{h,\gamma}dV_\gamma=\sum\gamma^{-2}_J|u_{J\lambda}|^2dV_\omega\leq|u|^2_{h,\omega}dV_\omega.
    \end{align*}

    From Lemma \ref{calculate (p,q)-forms}, we have that 
    \begin{align*}
        \langle A^{p,n}_{E,h,\gamma}u,u\rangle_\gamma&=\sum_{j\in J}c'_{jj\lambda\mu}u'_{J\lambda}\overline{u}'_{J\mu}+\sum_{j\ne k,L\setminus j=J\setminus k}c'_{jk\lambda\mu}u'_{L\lambda}\overline{u}'_{J\mu}\varepsilon(k,J)\varepsilon(j,L)\\
        &=\sum_{L\setminus j=J\setminus k}c'_{jk\lambda\mu}u'_{L\lambda}\overline{u}'_{J\mu}\varepsilon(k,J)\varepsilon(j,L)\\
        &=\gamma_N^{-2}\sum_{L\setminus j=J\setminus k}c_{jk\lambda\mu}u_{L\lambda}\overline{u}_{J\mu}\varepsilon(k,J)\varepsilon(j,L)/(\gamma_j\gamma_k\gamma_L\gamma_J)\\
        &=\gamma_N^{-2}\sum_I\sum_{L\setminus j=J\setminus k}c_{jk\lambda\mu}u_{L\lambda}\overline{u}_{J\mu}\varepsilon(k,J)\varepsilon(j,L)/(\gamma_j\gamma_k\gamma_I)^2 \quad (I\!:=\!L\!\setminus\! j\!=\!J\!\setminus\! k)\\
        &=\gamma_N^{-2}\sum_I\gamma_I^2\sum_{L\setminus j=J\setminus k}c_{jk\lambda\mu}u_{L\lambda}\overline{u}_{J\mu}\varepsilon(k,J)\varepsilon(j,L)/(\gamma_j\gamma_k\gamma_I^2)^2\\
        &=\gamma_N^{-2}\sum_I\gamma_I^2\sum_{L\setminus j=J\setminus k}c_{jk\lambda\mu}u_{L\lambda}\overline{u}_{J\mu}\varepsilon(k,J)\varepsilon(j,L)/(\gamma_L^2\gamma_J^2)\\
        &\geq \gamma_N^{-2}\sum_{L\setminus j=J\setminus k}c_{jk\lambda\mu}u_{L\lambda}\overline{u}_{J\mu}\varepsilon(k,J)\varepsilon(j,L)/(\gamma_L^2\gamma_J^2)\\
        &=\gamma_N^2\sum_{L\setminus j=J\setminus k}c_{jk\lambda\mu}u_{L\lambda}\overline{u}_{J\mu}\varepsilon(k,J)\varepsilon(j,L)/(\gamma_L^2\gamma_J^2\gamma_N^4)\\
        &=\gamma_N^2\langle A^{p,n}_{E,h,\omega}S_\gamma u,S_\gamma u\rangle_\omega
    \end{align*}
    where $S_\gamma$ is the operator defined by
    \begin{align*}
        S_\gamma u=\sum u_{J\lambda}\gamma_J^{-2}\gamma_N^{-2}dz_J\wedge d\overline{z}_N\otimes e_\lambda\in\Lambda^{p,n}T^*_{X,x_0}\otimes E_{x_0}.
    \end{align*}
    Therefore we obtain that $\, A^{p,n}_{E,h,\omega}>0~\, \Longrightarrow ~\,A^{p,n}_{E,h,\gamma}>0$.

    Hence for any $u,v\in\Lambda^{p,n}T^*_{X,x_0}\otimes E_{x_0}$ we have that
    \begin{align*}
        |\langle u,v\rangle_\gamma|^2=|\langle u,S_\gamma v\rangle_\omega|^2&\leq\langle(A^{p,n}_{E,h,\omega})^{-1} u,u\rangle_\omega\langle A^{p,n}_{E,h,\omega} S_\gamma v,S_\gamma v\rangle_\omega\\
        &\leq \gamma_N^{-2}\langle(A^{p,n}_{E,h,\omega})^{-1} u,u\rangle_\omega\langle A^{p,n}_{E,h,\gamma} v,v\rangle_\gamma,
    \end{align*}
    and the choice $v=(A^{p,n}_{E,h,\gamma})^{-1}u$ implies
    \begin{align*}
        \langle(A^{p,n}_{E,h,\gamma})^{-1}u,u\rangle_{h,\gamma}\gamma_N^2 \leq \langle(A^{p,n}_{E,h,\omega})^{-1}u,u\rangle_{h,\omega}.
    \end{align*}
    From the above and $dV_\gamma=\gamma_N^2dV_\omega$, this proof is completed. \qed\\

\begin{lemma}\label{uniformity of (p,q)-forms not dV}
    Let $X$ be a complex manifold and $(E,h)$ be a holomorphic Hermitian vector bundle over $X$.
    Let $\omega, \gamma$ be Hermitian metrics on $X$ such that $\gamma\geq\omega$. For any $u\in\Lambda^{p,q}T^*_X\otimes E$, we have that $|u|^2_{h,\gamma}\leq|u|^2_{h,\omega}$.
\end{lemma}

\begin{proof}
    Let notation be the same as one in the proof of Lemma \ref{uniformity of (p,n)-forms}.
    Then for any $u\in\Lambda^{p,q}T^*_{X,x_0}\otimes E_{x_0}$, we can write
    \begin{align*}
        u=\sum_{J,K,\lambda} u_{JK\lambda}dz_J\wedge d\overline{z}_K\otimes e_\lambda=\sum_{J,K,\lambda} u'_{JK\lambda}dw_J\wedge d\overline{w}_K\otimes e_\lambda
    \end{align*}
    with $u'_{JK\lambda}=u_{JK\lambda}/\gamma_J\gamma_K$. Hence, we have that 
    \begin{align*}
        |u|^2_{h,\gamma}=\sum|u'_{JK\lambda}|^2=\sum|u_{JK\lambda}|^2/(\gamma_J\gamma_K)^2\leq\sum|u_{JK\lambda}|^2=|u|^2_{h,\omega}.
    \end{align*}
    From the above, this proof is completed.
\end{proof}

Using Lemma \ref{uniformity of (p,n)-forms} and \ref{uniformity of (p,q)-forms not dV}, we obtain the following proposition.

\begin{proposition}\label{L2-estimate for (p,n)-forms on Stein}
    Let $S$ be a Stein manifold of dimension $n$ and $\omega$ be a \kah metric on $S$. Let $\varphi$ be a strictly plurisubharmonic function on $S$.
    We assume that 
    \begin{align*}
        \idd\varphi\geq\varepsilon\omega
    \end{align*}
    in the sense of currents for some $\varepsilon>0$. Then for any $f\in L^2_{p,n}(S,\varphi,\omega)$ satisfying $\overline{\partial}f=0$, there exists $u\in L^2_{p,n-1}(S,\varphi,\omega)$ such that $\overline{\partial}u=f$ and 
    \begin{align*}
        \int_S|u|^2_\omega e^{-\varphi}dV_\omega\leq\frac{1}{p\varepsilon}\int_S|f|_\omega^2e^{-\varphi}dV_\omega.
    \end{align*}
\end{proposition}

\begin{proof}
    We may assume that $S$ is a submanifold of $\mathbb{C}^{N}$.  
    By the theorem of Docquier and Grauert, there exists an open neighborhood $W \subset \mathbb{C}^{N}$ of $S$ and a holomorphic retraction $\mu: W \to S$ (cf.\,ChapterV\,of\,\cite{Hor90}).
    Let $\rho:\mathbb{C}^N\rightarrow \mathbb{R}_{\geq0}$ be a smooth function depending only on $|z|$ such that $\mathrm{supp}\,\rho\subset \mathbb{B}^N$ and that $\int_{\mathbb{C}^N} \rho(z)dV=1$, where $\mathbb{B}^N$ is the unit ball. 
    Define $\rho_\varepsilon(z)=(1/\varepsilon^{2n})h(z/\varepsilon)$ for $\varepsilon>0$. Let $S^\nu:=\{z\in S\mid d_N(z,S^c)>1/\nu\}$ be a subset of $S\subset \mathbb{C}^N$.
    For any plurisubharmonic function $\alpha$ on $S$ we define the function $\alpha_\nu:=\alpha\ast\rho_{1/\nu}$. Then $\alpha_\nu$ is a smooth plurisubharmonic function on $S^\nu$.
    
    Let $U$ be a open subset and $\Omega$ be a local \kah potential of $\omega$ on $U$, i.e. $\Omega$ satisfies $\idd\Omega=\omega$. 
    By the assumption, we get $\idd(\varphi-\varepsilon\Omega)=i\Theta_{L,h}-\varepsilon\omega\geq0$ in the sense of currents.
    Then the function $(\varphi-\varepsilon\Omega)_{\nu}=\varphi_\nu-\varepsilon\Omega_\nu$ is a smooth plurisubharmonic function defined on $U^\nu$.
    Since $\Omega_\nu$ is strictly plurisubharmonic, $\varphi_\nu$ also is a smooth strictly plurisubharmonic function on $S^\nu$ and satisfies the following condition
    \begin{align*}
        \idd\varphi_\nu\geq\varepsilon\idd\Omega_\nu\geq\varepsilon_\nu\omega, 
    \end{align*}
    where $(\varepsilon_\nu)_{\nu\in\mathbb{N}}$ is a positive number sequence such that $0<\varepsilon/2<\varepsilon_\nu\nearrow\varepsilon,\,\,(\nu\to+\infty)$.
    Let $\varphi_\infty:=\lim_{\nu\to+\infty}\varphi_\nu$ then $\varphi_\infty$ is a plurisubharmonic function on $S$ such that $\varphi_\infty=\varphi~\,a.e.$ and a smooth functions sequence $(\varphi_\nu)_{\nu\in\mathbb{N}}$ is decreasing to $\varphi_\infty$.

    Since Stein-ness of $S$, there exists a smooth exhaustive plurisubharmonic function $\psi$ on $S$. We can assume that $\sup_S\psi=+\infty$.
    For any number $c<\sup_S\psi=+\infty$, we define the sublevel sets $S_c:=\{z\in S\mid \psi(z)<c\}$ which is Stein. Fixed $j\in\mathbb{N}$. There exists $\nu_0\in\mathbb{N}$ such that for any integer $\nu\geq\nu_0$, $S_j\subset\subset S^{\nu_0} \subset\subset S^\nu$. 
    From Stein-ness of $S_j$, there exists a complete \kah metric $\widehat{\omega}_j$ on $S_j$. Then we define the complete \kah metric $\omega_\delta:=\omega+\delta\widehat{\omega}_j>\omega$ on $S_j$ for $\delta>0$.

    For any $\nu\geq\nu_0$ and any $v\in \Lambda^{p,n}T^*_{S_j}$, we obtain
    \begin{align*}
        \langle[\idd\varphi_\nu,\Lambda_\omega]v,v\rangle_\omega\geq\langle[\varepsilon_\nu\omega,\Lambda_\omega]v,v\rangle_\omega=p\varepsilon_\nu|v|^2 ~\,\,\mathrm{and}~\,\,A^{p,n}_{e^{-\varphi_\nu},\omega}=[\idd\varphi_\nu,\Lambda_\omega]>0.
    \end{align*}
    From this and Lemma \ref{uniformity of (p,n)-forms}, we have that $A^{p,n}_{e^{-\varphi_\nu},\omega_\delta}=[\idd\varphi_\nu,\Lambda_{\omega_\delta}]>0$ and
    \begin{align*}
        \int_{S_j}\langle[\idd\varphi_\nu,\Lambda_{\omega_\delta}]^{-1}f,f\rangle_{\omega_\delta} e^{-\varphi_\nu}dV_{\omega_\delta}&\leq \int_{S_j}\langle[\idd\varphi_\nu,\Lambda_\omega]^{-1}f,f\rangle_\omega e^{-\varphi_\nu}dV_\omega\\
        &\leq\frac{1}{p\varepsilon_\nu}\int_{S_j}|f|_\omega^2e^{-\varphi_\nu}dV_\omega\\
        &\leq\frac{1}{p\varepsilon_\nu}\int_{S_j}|f|_\omega^2e^{-\varphi}dV_\omega\\
        &\leq\frac{2}{p\varepsilon}\int_{S}|f|_\omega^2e^{-\varphi}dV_\omega<+\infty.
    \end{align*}

    For any two Hermitian metrics $\gamma_1,\gamma_2$ and any locally integrable function $\Phi\in\mathcal{L}_{loc}$, we define the Hilbert space $L^2_{p,q}(S,\Phi,\gamma_1,\gamma_2)$ of $(p,q)$-forms $g$ on $S$ with measurable coefficients such that 
    \begin{align*}
        \int_S|g|^2_{\gamma_1}e^{-\Phi}dV_{\gamma_2}<+\infty.
    \end{align*}
    Here there exists a positive smooth function $\tilde{\gamma}\in\mathcal{E}(S,\mathbb{R}_{>0})$ such that $dV_{\gamma_2}=\tilde{\gamma}dV_{\gamma_1}$ then we have that $L^2_{p,q}(S,\Phi,\gamma_1,\gamma_2)=L^2_{p,q}(S,\Phi-\log\tilde{\gamma},\gamma_1)$.

    Thanks to H\"ormander's $L^2$-estimate for smooth Hermitian metric with weight $\varphi_\nu$ and complete \kah metric $\omega_\delta$,
    we get a solution $u_{j,\nu,\delta}\in L^2_{p,n-1}(S_j,\varphi_\nu,\omega_\delta)\subset L^2_{p,n-1}(S_j,\varphi_\nu,\omega_\delta,\omega)$ of $\overline{\partial}u_{j,\nu,\delta}=f$ on $S_j$ such that
    \begin{align*}
        \int_{S_j}|u_{j,\nu,\delta}|^2_{\omega_\delta}e^{-\varphi_\nu}dV_\omega\leq\int_{S_j}|u_{j,\nu,\delta}|^2_{\omega_\delta}e^{-\varphi_\nu}dV_{\omega_\delta}&\leq\int_{S_j}\langle[\idd\varphi_\nu,\Lambda_{\omega_\delta}]^{-1}f,f\rangle_{\omega_\delta} e^{-\varphi_\nu}dV_{\omega_\delta}\\
        &\leq\int_{S_j}\langle[\idd\varphi_\nu,\Lambda_\omega]^{-1}f,f\rangle_\omega e^{-\varphi_\nu}dV_\omega\\
        &\leq\frac{2}{p\varepsilon}\int_{S}|f|_\omega^2e^{-\varphi}dV_\omega<+\infty.
    \end{align*}

    For fixed integer $\lambda_1\geq1$, $(u_{j,\nu,1/\lambda})_{\lambda_1\leq\lambda\in\mathbb{N}}$ forms a bounded sequence in $L^2_{p,n-1}(S_j,\varphi_\nu,\omega_{1/\lambda_1},\omega)$ due to the monotonicity of $|\bullet|^2_{\omega_{1/\lambda}}$, i.e. Lemma \ref{uniformity of (p,q)-forms not dV}.
    Therefore we can obtain a weakly convergent subsequence in $L^2_{p,n-1}(S_j,\varphi_\nu,\omega_{1/\lambda_1},\omega)$. By using a diagonal argument, we get a subsequence $(u_{j,\nu,\lambda_k})_{k\in\mathbb{N}}$ of $(u_{j,\nu,1/\lambda})_{\lambda\geq\lambda_1}$
    converging weakly in $L^2_{p,n-1}(S_j,\varphi_\nu,\omega_{1/\lambda_1},\omega)$ for any $\lambda_1$, where $u_{j,\nu,\lambda_k}\in L^2_{p,n-1}(S_j,\varphi_\nu,\omega_{1/\lambda_k},\omega)\subset L^2_{p,n-1}(S_j,\varphi_\nu,\omega_{1/\lambda_1},\omega)$. 
    We denote by $u_{j,\nu}$ the weak limit of $(u_{j,\nu,\lambda_k})_{k\in\mathbb{N}}$. Then $u_{j,\nu}$ satisfies $\overline{\partial}u_{j,\nu}=f$ on $S_j$ and 
    \begin{align*}
        \int_{S_j}|u_{j,\nu}|^2_{\omega_{\lambda_k}}e^{-\varphi_\nu} dV_\omega\leq\int_{S_j}\langle[\idd\varphi_\nu,\Lambda_\omega]^{-1}f,f\rangle_\omega e^{-\varphi_\nu}dV_\omega
    \end{align*}
    for each $k\in\mathbb{N}$. Taking weak limit $k\to+\infty$ and using the monotone convergence theorem, we have the following estimate
    \begin{align*}
        \int_{S_j}|u_{j,\nu}|^2_\omega e^{-\varphi_\nu} dV_\omega&\leq \int_{S_j}\langle[\idd\varphi_\nu,\Lambda_\omega]^{-1}f,f\rangle_\omega e^{-\varphi_\nu}dV_\omega\\
        &\leq\frac{1}{p\varepsilon_\nu}\int_{S_j}|f|_\omega^2e^{-\varphi}dV_\omega\leq\frac{2}{p\varepsilon}\int_{S}|f|_\omega^2e^{-\varphi}dV_\omega<+\infty,
    \end{align*}
    i.e. $u_{j,\nu}\in L^2_{p,n-1}(S_j,\varphi_\nu,\omega)$. 
    For fixed $\nu_1\geq\nu_0$, $(u_{j,\nu})_{\nu\geq\nu_1}$ forms a bounded sequence in $L^2_{p,n-1}(S_j,\varphi_{\nu_1},\omega)$ due to the monotonicity of $(\varphi_\nu)_{\nu\in\mathbb{N}}$.
    Repeating the above argument and taking the weak limit $\nu\to+\infty$, we get a solution $u_j\in L^2_{p,n-1}(S_j,\varphi,\omega)$ of $\overline{\partial}u_j=f$ on $S_j$ such that 
    \begin{align*}
        p\varepsilon_\nu\int_{S_j}|u_j|^2_\omega e^{-\varphi_\nu}dV_\omega\leq\int_{S}|f|_\omega^2e^{-\varphi}dV_\omega,
    \end{align*}
    for each $\nu\in\mathbb{N}$. Taking weak limit $\nu\to+\infty$ and using the monotone convergence theorem, we have the following estimate
    \begin{align*}
        p\varepsilon\int_{S_j}|u_j|^2_\omega e^{-\varphi}dV_\omega\leq\int_{S}|f|_\omega^2e^{-\varphi}dV_\omega.
    \end{align*}

    Finally, repeating the above argument and taking the weak limit $j\to+\infty$, we get a solution $u\in L^2_{p,n-1}(S,\varphi,\omega)$ of $\overline{\partial}u=f$ on $S$ such that 
    \begin{align*}
        \int_S|u|^2_\omega e^{-\varphi}dV_\omega\leq\frac{1}{p\varepsilon}\int_{S}|f|_\omega^2e^{-\varphi}dV_\omega.
    \end{align*}
    From the above, this proof is completed.
\end{proof}

Then, use Proposition \ref{L2-estimate for (p,n)-forms on Stein} to prove Theorem \ref{line bdl big L2-estimate for (p,n)-forms}.

\vspace{3mm}

$\textit{Proof of Theorem \ref{line bdl big L2-estimate for (p,n)-forms}}.$
    By Serre's GAGA, there exists a hypersurface $H\subset X$ such that $X\setminus H$ is Stein and $L$ is trivial over $X\setminus H$. From Proposition \ref{L2-estimate for (p,n)-forms on Stein}, for any $\overline{\partial}$-closed $f\in L^2_{p,n}(X,L,h,\omega)$ 
    there exists $u\in L^2_{p,n-1}(X\setminus H,-\log\mathrm{det}\,h,\omega)=L^2_{p,n-1}(X\setminus H,L,h,\omega)$ such that $\overline{\partial}u=f$ and 
    \begin{align*}
        \int_{X\setminus H}|u|^2_{h,\omega}dV_\omega&\leq\frac{1}{p\varepsilon}\int_{X\setminus H}|f|^2_{h,\omega}dV_\omega\\
        &\leq\frac{1}{p\varepsilon}\int_X|f|^2_{h,\omega}dV_\omega<+\infty.
    \end{align*}

    Letting $u=0$ on $H$, we have that $u\in L^2_{p,n-1}(X,L,h,\omega)$, $\overline{\partial}u=f$ on $X$ and 
    \begin{align*}
        \int_X|u|^2_{h,\omega}dV_\omega&\leq\frac{1}{p\varepsilon}\int_X|f|^2_{h,\omega}dV_\omega,
    \end{align*}
    from the following lemma. \qed

\begin{lemma}\label{Ext d-equation for hypersurface}$\mathrm{(cf.\,[Ber10,\,Lemma\,5.1.3])}$
    Let $X$ be a complex manifold and $H$ be a hypersurface in $X$. Let $u$ and $f$ be (possibly bundle valued) forms in $L^2_{loc}$ of $X$ satisfying $\overline{\partial}u=f$ on $X\setminus H$.
    Then the same equation holds on $X$ (in the sense of distributions).
\end{lemma}

Finally, we prove Theorem \ref{Ext Nadel V-thm} using Theorem \ref{line bdl big L2-estimate for (p,n)-forms} and the following lemma and theorem.

\begin{lemma}\label{Regularity of currents for (p,0)-forms}$(\mathrm{Dolbeault}$-$\mathrm{Grothendieck \,lemma},\, \mathrm{cf.\,[Dem}$-$\mathrm{book,\,ChapterI]} )$ 
    Let $T$ be a current of type $(p,0)$ on some open subset $U\subset \mathbb{C}^n$. If $T$ is $\overline{\partial}$-closed then it is a holomorphic differential form, i.e. a smooth differential form with holomorphic coefficients.
\end{lemma}

\begin{theorem}\label{Hor Thm 4.4.2}$(\mathrm{cf.\,[Hor90,\, Theorem\, 4.4.2]})$
    Let $\Omega$ be a pseudoconvex open set in $\mathbb{C}^n$ and $\varphi$ be any plurisubharmonic function in $\Omega$.
    For any $f\in L^2_{p,q+1}(\Omega,\varphi)$ with $\overline{\partial}f=0$ there exists a solution $u\in L^2_{p,q}(\Omega,\mathrm{loc})$ of the equation $\overline{\partial}u=f$ such that 
    \begin{align*}
        \int_\Omega|u|^2e^{-\varphi}(1+|z|^2)^{-2}dV_\Omega\leq\int_\Omega|f|^2e^{-\varphi}dV_\Omega.
    \end{align*}
\end{theorem}

\vspace{1mm}

$\textit{Proof of Theorem \ref{Ext Nadel V-thm}}.$
    We define the subsheaf $\mathscr{L}^{p,q}_{L,h}$ of germs of $(p,q)$-forms $u$ with values in $L$ and with measurable coefficients such that both $|u|^2_{h}$ and $|\overline{\partial}u|^2_h$ are locally integrable.
    And we consider the following sheaves sequence:
    \begin{align*}
        \xymatrix{
         0 \ar[r] & \mathrm{ker}\,\overline{\partial}_0 \hookrightarrow \mathscr{L}^{p,0}_{L,h} \ar[r]^-{\overline{\partial}_0} & \mathscr{L}^{p,1}_{L,h} \ar[r]^-{\overline{\partial}_1} & \cdots \ar[r]^-{\overline{\partial}_{n-1}} & \mathscr{L}^{p,n}_{L,h} \ar[r] & 0.
    }
    \end{align*}

    For any $x_0\in X$, there exists a bounded Stein open neighborhood $\Omega$ of $x_0$ such that $L|_\Omega$ is trivial. Then $-\log h$ is strictly plurisubharmonic function on $\Omega$ and $L^2_{p,q}(\Omega,L,h,\omega)=L^2_{p,q}(\Omega,-\log h,\omega)$.
    From Theorem \ref{Hor Thm 4.4.2}, for any $f\in L^2_{p,q}(\Omega,-\log h,\omega)$ with $\overline{\partial}f=0$ there exists a solution $u\in L^2_{p,q-1}(\Omega,\mathrm{loc})$ of the equation $\overline{\partial}u=f$ such that 
    \begin{align*}
        \inf_{z\in\Omega}\frac{1}{(1+|z|^2)^2}\int_\Omega|u|^2e^{\log h}dV_\omega\leq\int_\Omega|u|^2e^{\log h}(1+|z|^2)^{-2}dV_\omega\leq\int_\Omega|f|^2e^{\log h}dV_\omega<+\infty.
    \end{align*}
    Since boundedness of $\Omega$, we get $0<\inf_{z\in\Omega}(1+|z|^2)^{-2}$ and $u\in L^2_{p,q-1}(\Omega,-\log h,\omega)=L^2_{p,q-1}(\Omega,L,h,\omega)$.
    Then we have that the above sheaves sequence is exact.

    From Lemma \ref{Regularity of currents for (p,0)-forms}, the kernel of $\overline{\partial}_0$ consists of all germs of holomorphic $(p,0)$-forms with values in $L$ which satisfy the integrability condition
    and we have that $\mathrm{ker}\,\overline{\partial}_0=\Omega_X^p\otimes L\otimes \mathscr{I}(h)$.
    In fact, for any locally open subset $U\subset \mathbb{C}^n$ we obtain
    \begin{align*}
        f\in\mathrm{ker}\,\overline{\partial}_0(U) &\iff f=\sum f_Idz_I\in H^0(U,\Omega^p_X\otimes L)~\,\, \mathrm{such~ that}  \\
        & \int_U|f|^2_{h,\omega}dV_\omega=\int_U|f|^2e^{\log h}dV_\omega=\sum\int_U|f_I|^2e^{\log h}dV_\omega<+\infty.
    \end{align*}
    Therefore any $f_I\in H^0(U,\mathbb{C})$ satisfy the condition $f_I\in\mathscr{I}(h)(U)$.

    Since acyclicity of each $\mathscr{L}^{p,q}_{L,h}$, we obtain that 
    \begin{align*}
        H^q(X,\Omega_X^p\otimes L\otimes \mathscr{I}(h))\cong H^q(\Gamma(X,\mathscr{L}^{p,\bullet}_{L,h})).
    \end{align*}
    By Theorem \ref{line bdl big L2-estimate for (p,n)-forms}, we conclude that $H^n(\Gamma(X,\mathscr{L}^{p,\bullet}_{L,h}))=0$.    \qed

\vspace{2mm}

From the Demailly-Nadel vanishing theorem and Theorem \ref{Ext Nadel V-thm}, we get the following results (= extension of the Demailly-Nadel vanishing theorem) immediately:

Let $X$ be a projective manifold of dimension $n$ equipped with a \kah metric $\omega$ on $X$. Let $L$ be a holomorphic line bundle on $X$ equipped with a singular Hermitian metric $h$.
    We assume that 
    \begin{align*}
        i\Theta_{L,h}\geq\varepsilon\omega
    \end{align*}
    in the sense of currents for some $\varepsilon>0$. Then we have that
    \begin{align*}
        H^p(X,K_X\otimes L\otimes\mathscr{I}(h))&=0,\\
        H^n(X,\Omega_X^p\otimes L\otimes\mathscr{I}(h))&=0
    \end{align*}
    for $p>0$.

\begin{remark}
The above extension of the Demailly-Nadel vanishing theorem cannot be extended to the same bidegree $(p,q)$ with $p+q>n$ as the Nakano-Akizuki-Kodaira type vanishing theorem.
\end{remark}

In fact, Ramanujam has given in the following counterexample to the extension of the Nakano-Akizuki-Kodaira type vanishing theorem to nef and big line bundles.

\vspace{2mm}

$\mathbf{Counterexample.}~\,\mathrm{(cf.\,[Ram72],\,[Dem}$-$\mathrm{book,\,ChapterVII])}$
    Let $X$ be a blown up of one point in $\mathbb{P}^n$ and $\pi:X\to\mathbb{P}^n$ be the natural morphism. Clearly the line bundle $\pi^*\mathcal{O}_{\mathbb{P}^n}(1)$ is nef and big.
    Then we have the following non-vanishing cohomologies:
    \begin{align*}
        H^{p,p}(X,\pi^*\mathcal{O}_{\mathbb{P}^n}(1))\ne0 \quad \mathrm{for}\quad 0\leq p\leq n-1.
    \end{align*}

And, from the analytical characterization of nef and big line bundles (see \cite{Dem90}), there exist a singular Hermitian metric $h_{\pi^*O(1)}$ on $\pi^*\mathcal{O}_{\mathbb{P}^n}(1)$ such that 
$\mathscr{I}(h_{\pi^*O(1)})=\mathscr{O}_X$ and $i\Theta_{\pi^*\mathcal{O}_{\mathbb{P}^n}(1),h_{\pi^*O(1)}}\geq\varepsilon\omega$ in the sense of currents for some $\varepsilon>0$, where $\omega$ is a \kah metric on $X$.
Then we get the following counterexample:
\begin{align*}
    H^p(X,\Omega^p_X\otimes \pi^*\mathcal{O}_{\mathbb{P}^n}(1)\otimes \mathscr{I}(h_{\pi^*O(1)}))\cong H^{p,p}(X,\pi^*\mathcal{O}_{\mathbb{P}^n}(1))\ne0 \quad \mathrm{for}\quad 0\leq p\leq n-1.
\end{align*}

\section{Smooth Hermitian metrics and dual Nakano positivity}

Let $(X,\omega)$ be a complex manifold of complex dimension $n$ equipped with a Hermitian metric $\omega$ on $X$ and $(E,h)$ be a holomorphic Hermitian vector bundle of rank $r$ over $X$.
Let $D=D'+\overline{\partial}$ be the Chern connection of $(E,h)$, and $\Theta_{E,h}=[D',\overline{\partial}]=D'\overline{\partial}+\overline{\partial}D'$ be the Chern curvature tensor. 
Let $(U,(z_1,\cdots,z_n))$ be local coordinates. Denote by $(e_1,\cdots,e_r)$ an orthonormal frame of $E$ over $U\subset X$, and
\begin{align*}
    i\Theta_{E,h,x_0}=i\sum_{j,k}\Theta_{jk}dz_j\wedge d\overline{z}_k=i\sum_{j,k,\lambda,\mu}c_{jk\lambda\mu}dz_j\wedge d\overline{z}_k\otimes e^*_\lambda\otimes e_\mu,~\,\, \overline{c}_{jk\lambda\mu}=c_{kj\mu\lambda}.
\end{align*}
To $i\Theta_{E,h}$ corresponds a natural Hermitian form $\theta_{E,h}$ on $T_X\otimes E$ defined by 
\begin{align*}
    \theta_{E,h}(u):=\theta_{E,h}(u,u)&=\sum c_{jk\lambda\mu}u_{j\lambda}\overline{u}_{k\mu},~\,\, u=\sum u_{j\lambda}\frac{\partial}{\partial z_j}\otimes e_\lambda\in T_{X,x_0}\otimes E_x,\\
    \mathrm{i.e.}~\,\,\, \theta_{E,h}&=\sum c_{jk\lambda\mu}(dz_j\otimes e^*_\lambda)\otimes\overline{(dz_k\otimes e^*_\mu)}.
\end{align*}

\begin{definition}
    Let $X$ be a complex manifold and $(E,h)$ be a holomorphic Hermitian vector bundle over $X$.
    \begin{itemize}
        \item $(E,h)$ is said to be $\it{Griffiths ~positive}$ (resp. $\it{Griffiths ~semi}$-$\it{positive}$) if 
                    for any $\xi\in T_{X,x}$, $\xi\ne0$ and $s\in E_x$, $s\ne0$, we have \[ \theta_{E,h}(\xi\otimes s,\xi\otimes s)>0~\,\,(\mathrm{resp}. \geq0). \]
                    We write $(E,h)>_{Grif}0$, i.e. $i\Theta_{E,h}>_{Grif}0$ (resp. $\geq_{Grif}0$) for Griffiths positivity (resp. semi-positivity).
        \item $(E,h)$ is said to be $\it{Nakano ~positive}$ (resp. $\it{Nakano ~semi}$-$\it{positive}$) if $\theta_{E,h}$ is positive (resp. semi-positive) definite as a Hermitian form on $T_X\otimes E$, 
                    i.e. for any $u\in T_X\otimes E$, $u\ne0$, we have \[ \theta_{E,h}(u,u)>0~\,\,(\mathrm{resp}. \geq0). \]
                    We write $(E,h)>_{Nak}0$, i.e. $i\Theta_{E,h}>_{Nak}0$ (resp. $\geq_{Nak}0$) for Nakano positivity (resp. semi-positivity).
    \end{itemize}
\end{definition}

We introduce another notion about Nakano-type positivity.

\begin{definition}\label{Def dual Nakano positive smooth}$(\mathrm{cf.\,[LSY13,\,Definition\,2.1],\,[Dem20,\,Section\,1]})$
    Let $X$ be a complex manifold of dimension $n$ and $(E,h)$ be a holomorphic Hermitian vector bundle of rank $r$ over $X$.
    $(E,h)$ is said to be $\it{dual~Nakano ~positive}$ (resp. $\it{dual~Nakano ~semi}$-$\it{positive}$) if $(E^*,h^*)$ is Nakano negative (resp. Nakano semi-negative).
\end{definition}

From definitions, we see immediately that if $(E,h)$ is Nakano positive or dual Nakano positive then $(E,h)$ is Griffiths positive. And there is an example of dual Nakano positive as follows.
Let $h_{FS}$ be the Fubini-Study metric on $T_{\mathbb{P}^n}$, then $(T_{\mathbb{P}^n},h_{FS})$ is dual Nakano positive and Nakano semi-positive (cf.\,[LSY13,\,Corollary\,7.3]). 
$(T_{\mathbb{P}^n},h_{FS})$ is easyly shown to be ample, but it is not Nakano positive.
In fact, if $(T_{\mathbb{P}^n},h_{FS})$ is Nakano positive then from the Nakano vanishing theorem (see \cite{Nak55}), we have that 
\begin{align*}
    H^{n-1,n-1}(\mathbb{P}^n,\mathbb{C})=H^{n-1}(\mathbb{P}^n,\Omega^{n-1}_{\mathbb{P}^n})=H^{n-1}(\mathbb{P}^n,K_{\mathbb{P}^n}\otimes T_{\mathbb{P}^n})=0.
\end{align*}
However, this contradicts $H^{n-1,n-1}(\mathbb{P}^n,\mathbb{C})=\mathbb{C}$.

Here, the following theorem is known, which expresses the relationship for the three positivity, i.e. Griffiths, Nakano and dual Nakano.

\begin{theorem}\label{smooth Grif then Nak for tensor}$(\mathrm{cf.\,[DS79,\,Theorem\,1],\,[LSY13,\,Theorem\,7.2]})$
    Let $h$ be a smooth Hermitian metric on $E$. If $(E,h)$ is Griffiths semi-positive then $(E\otimes\mathrm{det}\,E,h\otimes\mathrm{det}\,h)$ is Nakano semi-positive and dual Nakano semi-positive.
\end{theorem}

Let $\mathcal{E}^{p,q}(E)$ be the sheaf of germs of $\mathcal{C}^\infty$ sections of $\Lambda^{p,q}T^*_X\otimes E$ and $\mathcal{D}^{p,q}(E)$ be the space of $\mathcal{C}^\infty$ sections of $\Lambda^{p,q}T^*_X\otimes E$ with compact support on $X$.

Deng, Ning, Wang and Zhou introduced a positive notion of H\"ormander type in \cite{DNWZ20}, which is named as $\textit{the optimal}$ $L^p$-$\textit{estimate condition}$ and characterizes Nakano semi-positivity, i.e. $A^{n,1}_{E,h}\geq0$, for holomorphic vector bundles $(E,h)$. 
Then we introduced the following positive notion of H\"ormander type in \cite{Wat21}, which is an extension of the optimal $L^2$-estimate condition from $(n,1)$-forms to $(p,n)$-forms and which characterizes the condition $A^{p,n}_{E,h}\geq0$ (see Theorem \ref{(p,n)-condition iff A semi-posi}).

\begin{definition}\label{def (p,n)-condition}$\mathrm{(cf.\,[Wat22,\,Definition\,1.4]})$
    Let $(X,\omega)$ be a \kah manifold of dimension $n$ which admits a positive holomorphic Hermitian line bundle and $E$ be a holomorphic vector bundle over $X$ equipped with a (singular) Hermitian metric $h$.
    $(E,h)$ satisfies $\it{the}$ $(p,n)$-$L^2_\omega$-$\it{estimate~ condition}$ on $X$, 
    if for any positive holomorphic Hermitian line bundle $(A,h_A)$ on $X$ and for any $f\in\mathcal{D}^{p,n}(X,E\otimes A)$ with $\overline{\partial}f=0$,
    there is $u\in L^2_{p,n-1}(X,E\otimes A)$ satisfying $\overline{\partial}u=f$ and 
    \[ \int_X|u|^2_{h\otimes h_A,\omega}dV_\omega\leq\int_X\langle[i\Theta_{A,h_A}\otimes \mathrm{id}_E,\Lambda_\omega]^{-1}f,f\rangle_{h\otimes h_A,\omega} dV_\omega, \]
    provided that the right hand side is finite.

    And $(E,h)$ satisfies $\it{the}$ $(p,n)$-$L^2$-$\it{estimate~ condition}$ on $X$ if for any \kah metric $\tilde{\omega}$, $(E,h)$ satisfies $\it{the}$ $(p,n)$-$L^2_{\tilde{\omega}}$-$\it{estimate~ condition}$ on $X$
\end{definition}

\begin{theorem}\label{(p,n)-condition iff A semi-posi}$\mathrm{(cf.\,[Wat22,\,Theorem\,1.6]})$
    Let $(X,\omega)$ be a \kah manifold of dimension $n$ which admits a positive holomorphic Hermitian line bundle and $(E,h)$ be a holomorphic Hermitian vector bundle over $X$ and $p$ be a nonnegative integer.
    Then $(E,h)$ satisfies the $(p,n)$-$L^2_\omega$-estimate condition on $X$ if and only if $A^{p,n}_{E,h,\omega}\geq0$.
\end{theorem}

Here, as is well known, we know the following two facts about smooth Hermitian metrics $h$ on $E$: Let $(X,\omega)$ be a \kah manifold.
\begin{align*}
    A^{n,1}_{E,h,\omega}\geq0 ~\,(\mathrm{resp}. \leq0) \Longrightarrow A^{n,q}_{E,h,\omega}\geq0 ~\,(\mathrm{resp}. \leq0)\,~ \mathrm{for\,all}\,~ q\geq1 \quad&(\mathrm{see \,[Dem82]}), \\
    (E,h)\geq_{Nak}\!0 ~\,(\mathrm{resp}. \leq_{Nak}\!0) \iff A^{n,1}_{E,h,\omega}\geq0 ~\,(\mathrm{resp}. \leq0) \,~\quad\qquad&(\mathrm{see\, [DNWZ20]}). 
\end{align*}
Therefore, from these two facts, Lemma \ref{characterization of A(p,q)} and the definition of Nakano semi-positivity, we obtain the following characterizations:
\begin{align*}
    (E,h) ~\textit{is Nakano semi-positive} &\iff A^{n,q}_{E,h,\omega}\geq0 ~\,\textit{for\,all}~\,q\geq1, \tag*{(a)}\\
    (E,h) ~\textit{is dual Nakano semi-positive} &\iff A^{p,n}_{E,h,\omega}\geq0 ~\,\textit{for\,all}~\,p\geq1. \tag*{(b)}
\end{align*}

\begin{lemma}\label{characterization of A(p,q)}$\mathrm{(cf.\,[Wat22,\,Theorem\,2.3\,\,and\,\,2.5])}$
    Let $(X,\omega)$ be a Hermitian manifold and $(E,h)$ be a holomorphic vector bundle over $X$. We have that 
    \begin{align*}
        A^{n-p,n-q}_{E^*,h^*,\omega}\geq0~(\textit{resp.} \leq0) \iff A^{p,q}_{E,h,\omega}\geq0~(\textit{resp.} \leq0) \iff A^{n-q,n-p}_{E,h,\omega}\leq0~(\textit{resp.} \geq0).
    \end{align*}
\end{lemma}

By using the second condition (b), we show the following theorem which is already known as $(n,q)$-forms in the case of Nakano semi-positive.

\begin{theorem}\label{L^2-estimate completeness (p,n)-forms}
    Let $(X,\widehat{\omega})$ be a complete \kah manifold, $\omega$ be another \kah metric which is not necessarily complete and $(E,h)$ be a dual Nakano semi-positive vector bundle. 
    Then for any $\overline{\partial}$-closed $f\in L^2_{p,n}(X,E,h,\omega)$ there exists $u\in L^2_{p,n-1}(X,E,h,\omega)$ satisfies $\overline{\partial}u=f$ and 
    \begin{align*}
        \int_X|u|^2_{h,\omega}dV_{\omega}\leq\int_X\langle(A^{p,n}_{E,h,\omega})^{-1}f,f\rangle_{h,\omega}dV_{\omega},
    \end{align*}
    where we assume that the right-hand side is finite.
\end{theorem}

Furthermore, from the condition (b), Theorem \ref{(p,n)-condition iff A semi-posi} and \ref{L^2-estimate completeness (p,n)-forms}, we obtain the following characterization of dual Nakano semi-positivity by using $L^2$-estimates.

\begin{theorem}\label{smooth dual Nakano L2-condition}$\mathrm{(cf.\,[Wat22,\,Corollary\,4.5]})$
    Let $(X,\omega)$ be a \kah manifold of dimension $n$ which admits a positive holomorphic Hermitian line bundle and $(E,h)$ be a holomorphic Hermitian vector bundle over $X$. 
    Then $(E,h)$ satisfies the $(p,n)$-$L^2$-estimate condition for all $p\geq1$ if and only if $(E,h)$ is dual Nakano semi-positive.
\end{theorem}

\vspace{3mm}

$\textit{Proof of Theorem \ref{L^2-estimate completeness (p,n)-forms}}.$
    For any two Hermitian metrics $\gamma_1,\gamma_2$, we define the Hilbert space $L^2_{p,q}(X,E,h,\gamma_1,\gamma_2)$ of $(p,q)$-forms $g$ on $X$ with measurable coefficients such that 
    \begin{align*}
        \int_X|g|^2_{h,\gamma_1}dV_{\gamma_2}<+\infty.
    \end{align*}
    Here there exists a positive smooth function $\tilde{\gamma}\in\mathcal{E}(X,\mathbb{R}_{>0})$ such that $dV_{\gamma_2}=\tilde{\gamma}dV_{\gamma_1}$ then we have that $L^2_{p,q}(X,E,h,\gamma_1,\gamma_2)=L^2_{p,q}(X,E,\tilde{\gamma}h,\gamma_1)$.

    For every $\varepsilon>0$, the \kah metric $\omega_\varepsilon=\omega+\varepsilon\widehat{\omega}$ is complete. The idea of the proof is to apply the $L^2$-estimates to $\omega_\varepsilon$ and to let $\varepsilon$ tend to zero.
    It follows from Lemma \ref{uniformity of (p,n)-forms} and the equivalence condition of dual Nakano semi-positivity, i.e. $A^{p,n}_{E,h,\omega}\geq0$ for $p\geq1$ that 
    \begin{align*}
        \langle(A^{p,n}_{E,h,\omega_\varepsilon})^{-1}g,g\rangle_{h,\omega_\varepsilon}dV_{\omega_\varepsilon}\leq\langle(A^{p,n}_{E,h,\omega})^{-1}g,g\rangle_{h,\omega}dV_{\omega}
    \end{align*}
    for any $g\in\Lambda^{p,n}T^*_X\otimes E$. Thanks to H\"ormander's $L^2$-estimate, we get the solution $u_\varepsilon\in L^2_{p,n-1}(X,E,h,\omega_\varepsilon)\subset L^2_{p,n-1}(X,E,h,\omega_\varepsilon,\omega)$ of $\overline{\partial}u_\varepsilon=f$ such that 
    \begin{align*}
        \int_X|u|^2_{h,\omega_\varepsilon}dV_\omega\leq\int_X|u|^2_{h,\omega_\varepsilon}dV_{\omega_\varepsilon}\leq\int_X\langle(A^{p,n}_{E,h,\omega_\varepsilon})^{-1}f,f\rangle_{h,\omega_\varepsilon}dV_{\omega_\varepsilon}\leq\int_X\langle(A^{p,n}_{E,h,\omega})^{-1}f,f\rangle_{h,\omega}dV_{\omega},
    \end{align*}
    where $dV_\omega\leq dV_{\omega_\varepsilon}$.

    For fixed integer $j_0\geq1$, $(u_{1/j})_{j\in\mathbb{N}_{\geq j_0}}$ forms a bounded sequence in $L^2_{p,n-1}(X,E,h,\omega_{1/j_0},\omega)$ due to the monotonicity of $|\bullet|^2_{\omega_{1/j}}$, i.e. Lemma \ref{uniformity of (p,q)-forms not dV}.
    Therefore we can obtain a weakly convergent subsequence in $L^2_{p,n-1}(X,E,h,\omega_{1/j_0},\omega)$. By using a diagonal argument, we get a subsequence $(u_{j_k})_{k\in\mathbb{N}}$ of $(u_{1/j})_{j\in\mathrm{N}_{\geq j_0}}$
    converging weakly in $L^2_{p,n-1}(X,E,h,\omega_{1/j_0},\omega)$ for any $j_0$, where $u_{j_k}\in L^2_{p,n-1}(X,E,h,\omega_{1/j_k},\omega)\subset L^2_{p,n-1}(X,E,h,\omega_{1/j_0},\omega)$. 
    We denote by $u$ the weak limit of $(u_{j_k})_{k\in\mathbb{N}}$. Then $u$ satisfies $\overline{\partial}u=f$ and 
    \begin{align*}
        \int_X|u|^2_{h,\omega_{1/j_k}}dV_\omega\leq\int_X\langle(A^{p,n}_{E,h,\omega})^{-1}f,f\rangle_{h,\omega}dV_{\omega},
    \end{align*}
    for each $k\in\mathbb{N}$. Taking weak limit $k\to+\infty$ and using the monotone convergence theorem, we have the following estimate
    \begin{align*}
        \int_X|u|^2_{h,\omega}dV_\omega\leq\int_X\langle(A^{p,n}_{E,h,\omega})^{-1}f,f\rangle_{h,\omega}dV_{\omega}<+\infty,
    \end{align*}
    i.e. $u\in L^2_{p,n-1}(X,E,h,\omega)$.\qed

    \vspace{3mm}

Finally, we get the following proposition by applying and modifying Theorem \ref{smooth dual Nakano L2-condition}.

\begin{proposition}\label{characterization of smooth dual Nakano}
    Let $h$ be a smooth Hermitian metric on $E$. We consider the following conditions:
    \begin{itemize}
        \item [(1)] $h$ is dual Nakano semi-positive.
        \item [(2)] For any Stein coordinate $S$ such that $E|_S$ is trivial on $S$, any K\"ahler metric $\omega_S$ on $S$, any smooth strictly plurisubharmonic function $\psi$ on $S$, any integer $p\in\{1,\cdots,n\}$ 
        and any $\overline{\partial}$-closed $f\in L^2_{p,n}(S,E,he^{-\psi},\omega_S)$, there exists $u\in L^2_{p,n-1}(S,E,he^{-\psi},\omega_S)$ satisfying $\overline{\partial}u=f$ and 
        \begin{align*}
            \int_S|u|^2_{h,\omega_S}e^{-\psi}dV_{\omega_S}\leq\int_S\langle B^{-1}_{\psi,\omega_S}f,f\rangle_{h,\omega_S}e^{-\psi}dV_{\omega_S},
        \end{align*}
        provided the right-hand side is finite, where $B_{\psi,\omega_S}=[\idd\psi\otimes\mathrm{id}_E,\Lambda_{\omega_S}]$.
        \item [(3)] $(E,h)$ satisfies the $(p,n)$-$L^2$-estimate condition for all $p\geq1$.
    \end{itemize}
    Then two conditions $(1)$ and $(2)$ are equivalent. If $X$ admits a complete \kah metric $\omega$ and a positive holomorphic line bundle on $X$, the above three conditions are equivalent.
\end{proposition}

\begin{proof}
    First, we consider $(1) \Longrightarrow (2)$. We have that $i\Theta_{E,he^{-\psi}}=i\Theta_{E,h}+\idd\psi\otimes\mathrm{id}_E$ is dual Nakano positive on $S$ and
    \begin{align*}
        A^{p,n}_{E,he^{-\psi},\omega_S}=[i\Theta_{E,h},\Lambda_{\omega_S}]+[\idd\psi\otimes\mathrm{id}_E,\Lambda_{\omega_S}]=A^{p,n}_{E,h,\omega_S}+B_{\psi,\omega_S}\geq B_{\psi,\omega_S}>0~\,\mathrm{on }~S.
    \end{align*}
    Since Theorem \ref{L^2-estimate completeness (p,n)-forms}, for any $p\geq1$ and for any $\overline{\partial}$-closed $f\in L^2_{p,n}(S,E,he^{-\psi},\omega_S)$ there exists $u\in L^2_{p,n-1}(S,E,he^{-\psi},\omega_S)$ such that $\overline{\partial}u=f$ and 
    \begin{align*}
        \int_S|u|^2_{h,\omega}e^{-\psi}dV_{\omega_S}&\leq\int_S\langle(A^{p,n}_{E,he^{-\psi},\omega_S})^{-1}f,f\rangle_{h,\omega_S}e^{-\psi}dV_{\omega_S}\\
        &\leq\int_S\langle B^{-1}_{\psi,\omega_S}f,f\rangle_{h,\omega_S}e^{-\psi}dV_{\omega_S}.
    \end{align*}

    Next, we consider $(2) \Longrightarrow (1)$. From the condition $(2)$, for any very small Stein coordinate $S$, $(E,h)$ satisfies the $(p,n)$-$L^2_{\omega_S}$-estimate condition on $S$. Since Theorem \ref{smooth dual Nakano L2-condition}, we have that $A^{p,n}_{E,h,\omega_S}\geq0$ which is equivalent to dual Nakano semi-positive on $S$. 
    Since dual Nakano semi-positive is a local property, we get the condition $(1)$.

    Finally, we assume that $X$ admits a complete \kah metric $\omega$ and a positive holomorphic line bundle on $X$. From Theorem \ref{smooth dual Nakano L2-condition} and \ref{L^2-estimate completeness (p,n)-forms}, we have that $(3) \iff (1)$.
\end{proof}

\section{Singular Hermitian metrics and characterization of dual Nakano positivity}

In this section, we consider the case where a Hermitian metric of a holomorphic vector bundle has singularities.
First, for holomorphic vector bundles, we introduce the definition of singular Hermitian metrics $h$ and 
 the multiplier submodule sheaf $\mathscr{E}(h)$ of $\mathscr{O}(E)$ with respect to $h$ that is analogous to the multiplier ideal sheaf.

\begin{definition}$(\mathrm{cf.~[BP08,~Section~3],~[Rau15,~Definition~1.1]~and~[PT18,~Definition}$ $\mathrm{2.2.1]})$
    We say that $h$ is a $\it{singular~Hermitian~metric}$ on $E$ if $h$ is a measurable map from the base manifold $X$ to the space of non-negative Hermitian forms on the fibers satisfying $0<\mathrm{det}\,h<+\infty$ almost everywhere.
\end{definition}

\begin{definition}$(\mathrm{cf.\,[deC98,\,Definition\,2.3.1]})$
    Let $h$ be a singular Hermitian metric on $E$. We define the ideal subsheaf $\mathscr{E}(h)$ of germs of local holomorphic sections of $E$ as follows:
    \begin{align*}
        \mathscr{E}(h)_x:=\{s_x\in\mathscr{O}(E)_x\mid|s_x|^2_h~ \mathrm{is ~locally ~integrable ~around~} x\}.
    \end{align*}
\end{definition}

Moreover, we introduce the definitions of positivity and negativity, such as Griffiths and Nakano, for singular Hermitian metrics.

\begin{definition}\label{sing Griffiths def}$(\mathrm{cf.~[BP08,~Definition~3.1],~[Rau15,~Definition~1.2]~and~[PT18,~Def}$- 
    $\mathrm{inition~2.2.2]})$
    We say that a singular Hermitian metric $h$ is 
    \begin{itemize}
        \item [(1)] $\textit{Griffiths semi-negative}$ if $|u|_h$ is plurisubharmonic for any local holomorphic section $u\in\mathscr{O}(E)$ of $E$.
        \item [(2)] $\textit{Griffiths semi-positive}$ if the dual metric $h^*$ on $E^*$ is Griffiths semi-negative.
    \end{itemize}
\end{definition}

Let $h$ be a smooth Hermitian metric on $E$ and $u=(u_1,\cdots,u_n)$ be an $n$-tuple of holomorphic sections of $E$. We define $T^h_u$, an $(n-1,n-1)$-form through
\begin{align*}
    T^h_u=\sum^n_{j,k=1}(u_j,u_k)_h\widehat{dz_j\wedge d\overline{z}_k}
\end{align*}
where $(z_1,\cdots,z_n)$ are local coordinates on $X$, and $\widehat{dz_j\wedge d\overline{z}_k}$ denotes the wedge product of all $dz_i$ and $d\overline{z}_i$ expect $dz_j$ and $d\overline{z}_k$, 
multiplied by a constant of absolute value $1$, chosen so that $T_u$ is a positive form. Then a short computation yields that $(E,h)$ is Nakano semi-negative if and only if $T^h_u$ is plurisubharmonic in the sense that $\idd T^h_u\geq0$ (see \cite{Ber09},\,\cite{Rau15}).
In the case of $u_j=u_k=u$, $(E,h)$ is Griffiths semi-negative.

From the above, we introduce the definition of Nakano semi-negativity for singular Hermitian metrics.

\begin{definition}\label{def Nakano semi-negative as Raufi}$\mathrm{(cf.\,[Rau15,\,Section\,1]})$
    We say that a singular Hermitian metric $h$ on $E$ is Nakano semi-negative if the $(n-1,n-1)$-form $T^h_u$ is plurisubharmonic for any $n$-tuple of holomorphic sections $u=(u_1,\cdots,u_n)$.
\end{definition}

Here, since the dual of a Nakano negative bundle in general is not Nakano positive, we cannot define Nakano semi-positivity for singular Hermitian metrics as in the case of Griffiths semi-positive, 
but we naturally define dual Nakano semi-positivity for singular Hermitian metrics as follows.

\begin{definition}\label{def dual Nakano semi-posi sing}
    We say that a singular Hermitian metric $h$ on $E$ is dual Nakano semi-positive if the dual metric $h^*$ on $E^*$ is Nakano semi-negative.
\end{definition}

For Nakano semi-positivity of singular Hermitian metrics, we already know one definition in \cite{Ina21b}, which is based on the optimal $L^2$-estimate condition in \cite{HI20}, \cite{DNWZ20} and is equivalent to the usual definition for the smooth case.

\begin{definition}\label{def Nakano semi-posi sing}$\mathrm{(cf.\,[Ina22,\,Definition\,1.1]})$
    Assume that $h$ is a Griffiths semi-positive singular Hermitian metric. We say that $h$ is 
    \textit{Nakano semi-positive} if for any Stein coordinate $S$ such that $E|_S$ is trivial, any \kah metric $\omega_S$ on $S$,
    any smooth strictly plurisubharmonic function $\psi$ on $S$, any positive integer $q\in\{1,\cdots,n\}$ and any $\overline{\partial}$-closed $f\in L^2_{n,q}(S,E,he^{-\psi},\omega_S)$ 
    there exists $u\in L^2_{n,q-1}(S,E,he^{-\psi},\omega_S)$ satisfying $\overline{\partial}u=f$ and 
    \begin{align*}
        \int_S|u|^2_{h,\omega_S}e^{-\psi}dV_{\omega_S}\leq\int_S\langle B^{-1}_{\psi,\omega_S}f,f\rangle_{h,\omega_S}e^{-\psi}dV_{\omega_S},
    \end{align*}
    where $B_{\psi,\omega_S}=[\idd\psi\otimes\mathrm{id}_E,\Lambda_{\omega_S}]$. Here we assume that the right-hand side is finite.
\end{definition}

In \cite{Nad89}, Nadel proved that $\mathscr{I}(h)$ is coherent by using the H\"ormander $L^2$-estimate. 
After that, as holomorphic vector bundles case, Hosono and Inayama proved that $\mathscr{E}(h)$ is coherent if $h$ is Nakano semi-positive in the sense of singular as in Definition \ref{def Nakano semi-posi sing} in \cite{HI20} and \cite{Ina21b}.

For singular Hermitian metrics, we cannot always define the curvature currents with measure coefficients (see \cite{Rau15}). 
However, the above Definition \ref{def Nakano semi-posi sing} can be defined by not using the curvature currents of a singular Hermitian metric directly.
Therefore, by using these definitions, the following definition of strictly positivity for Griffiths and Nakano is known.

\begin{definition}\label{sing strictly positive of Grif and Nakano}$\mathrm{(cf.\,[Ina20,\,Definition\,2.6],\,[Ina22,\,Definition\,2.16])}$
    Let $(X,\omega_X)$ be a \kah manifold and $h$ be a singular Hermitian metric on $E$.
    \begin{itemize}
        \item We say that $h$ is $\textit{strictly Griffiths}~ \delta_{\omega_X}$-$\textit{positive}$ if for any open subset $U$ and any \kah potential $\varphi$ of $\omega_X$ on $U$, $he^{\delta\varphi}$ is Griffiths semi-positive on $U$.
        \item We say that $h$ is $\textit{strictly Nakano}~ \delta_{\omega_X}$-$\textit{positive}$ if for any open subset $U$ and any \kah potential $\varphi$ of $\omega_X$ on $U$, $he^{\delta\varphi}$ is Nakano semi-positive on $U$ in the sense of Definition \ref{def Nakano semi-posi sing}.
    \end{itemize}
\end{definition}

This definition for Nakano gives the following $L^2$-estimate theorem and establishes the singular-type Nakano vanishing theorem (Theorem \ref{sing type Nakano V-thm}) by using this $L^2$-estimate theorem.

\begin{theorem}\label{L2-estimate for Nakano positive}$\mathrm{(cf.\,[Ina22,\,Theorem\,1.4])}$
    Let $(X,\omega_X)$ be a projective manifold and a \kah metric on $X$ and $q$ be a positive integer.
    We assume that $(E,h)$ is strictly Nakano $\delta_{\omega_X}$-positive in the sense of Definition \ref{sing strictly positive of Grif and Nakano} on $X$. Then for any $\overline{\partial}$-closed $f\in L^2_{n,q}(X,E,h,\omega_X)$
    there exists $u\in L^2_{n,q-1}(X,E,h,\omega_X)$ satisfies $\overline{\partial}u=f$ and 
    \begin{align*}
        \int_X|u|^2_{h,\omega_X}dV_{\omega_X}\leq\frac{1}{\delta q}\int_X|f|^2_{h,\omega_X}dV_{\omega_X}.
    \end{align*}
\end{theorem}

\begin{theorem}\label{sing type Nakano V-thm}$\mathrm{(cf.\,[Ina22,\,Theorem\,1.5])}$
    Let $(X,\omega_X)$ be a projective manifold and a \kah metric on $X$. We assume that $(E,h)$ is strictly Nakano $\delta_{\omega_X}$-positive in the sense of Definition \ref{sing strictly positive of Grif and Nakano} on $X$.
    Then the $q$-th cohomology group of $X$ with coefficients in the sheaf of germs of holomorphic sections of $K_X\otimes\mathscr{E}(h)$ vanishes for $q>0$:
    \begin{align*}
        H^q(X,K_X\otimes\mathscr{E}(h))=0,
    \end{align*}
    where $\mathscr{E}(h)$ is the sheaf of germs of locally square integrable holomorphic sections of $E$ with respect to $h$.
\end{theorem}

Here, from the above discussion we consider dual Nakano positivity using $L^2$-estimates and its relation to Definition \ref{def dual Nakano semi-posi sing}.
For convenience, we say that $h$ is $L^2$-$\textit{type}$ $\textit{dual Nakano}$ $\textit{semi}\,$-$\textit{positive}$ if $h$ is a Griffiths semi-positive singular Hermitian metric and has the following $L^2$-estimates condition 
(\,=\,the singular case of condition $(2)$ in Proposition \ref{characterization of smooth dual Nakano}):

    For any Stein coordinate $S$ such that $E|_S$ is trivial, any \kah metric $\omega_S$ on $S$,
    any smooth strictly plurisubharmonic function $\psi$ on $S$, any positive integer $p\in\{1,\cdots,n\}$ and any $\overline{\partial}$-closed $f\in L^2_{p,n}(S,E,he^{-\psi},\omega_S)$ 
    there exists $u\in L^2_{p,n-1}(S,E,he^{-\psi},\omega_S)$ satisfying $\overline{\partial}u=f$ and 
    \begin{align*}
        \int_S|u|^2_{h,\omega_S}e^{-\psi}dV_{\omega_S}\leq\int_S\langle B^{-1}_{\psi,\omega_S}f,f\rangle_{h,\omega_S}e^{-\psi}dV_{\omega_S},
    \end{align*}
    where $B_{\psi,\omega_S}=[\idd\psi\otimes\mathrm{id}_E,\Lambda_{\omega_S}]$. Here we assume that the right-hand side is finite.

Then we obtain the following proposition that the natural Definition \ref{def dual Nakano semi-posi sing} satisfies the above condition analogous to Definition \ref{def Nakano semi-posi sing}.

\begin{proposition}\label{dual Nakano semi-posi sing then has L2-estimate condition}
    Assume that a singular Hermitian metric $h$ on $E$ is dual Nakano semi-positive. 
    Then $h$ is $L^2$-type dual Nakano semi-positive. 
\end{proposition}

This proof is in the next section. The above discussion of Nakano semi-positivity and this proposition show the usefulness of Definition \ref{def dual Nakano semi-posi sing}.
From Proposition \ref{characterization of smooth dual Nakano}, this $L^2$-estimates condition can be considered a natural extension of dual Nakano semi-positivity to singular Hermitian metrics and coincides with the usual definition if $h$ is smooth. 

Using Definition \ref{sing strictly positive of Grif and Nakano} as a reference, we introduce a definition of strictly positivity for dual Nakano as follows.

\begin{definition}\label{sing strictly positive of dual Nakano}
    Let $(X,\omega_X)$ be a \kah manifold and $h$ be a singular Hermitian metric on $E$.
    We say that $h$ is $\textit{strictly dual Nakano}~ \delta_{\omega_X}$-$\textit{positive}$ if for any open subset $U$ and any \kah potential $\varphi$ of $\omega_X$ on $U$, $he^{\delta\varphi}$ is dual Nakano semi-positive on $U$. 
\end{definition}

By using this definition, we get the following $L^2$-estimate theorem which is an extension of Theorem \ref{line bdl big L2-estimate for (p,n)-forms} to holomorphic vector bundles.

\begin{theorem}\label{L2-estimate for dual Nakano positive}
    Let $(X,\omega_X)$ be a projective manifold and a \kah metric on $X$ and $p$ be a positive integer.
    We assume that $(E,h)$ is strictly dual Nakano $\delta_{\omega_X}$-positive on $X$. Then for any $\overline{\partial}$-closed $f\in L^2_{p,n}(X,E,h,\omega_X)$
    there exists $u\in L^2_{p,n-1}(X,E,h,\omega_X)$ satisfies $\overline{\partial}u=f$ and 
    \begin{align*}
        \int_X|u|^2_{h,\omega_X}dV_{\omega_X}\leq\frac{1}{\delta p}\int_X|f|^2_{h,\omega_X}dV_{\omega_X}.
    \end{align*}
\end{theorem}

\begin{proof}
    By Serre's GAGA, there exists a proper Zariski open subset $S$ such that $E|_S$ is trivial over $S$ and $\omega_X$ is $\partial\overline{\partial}$-exact on $S$. We can take as a Stein open subset.
    We fix a \kah potential $\varphi$ of $\omega_X$ on $S$. Then we have that 
    \begin{align*}
        \langle B_{\delta\varphi,\omega_X}f,f\rangle_{h,\omega_X}
        &=\delta p|f|^2_{h,\omega_X},\\
        \langle B^{-1}_{\delta\varphi,\omega_X}f,f\rangle_{h,\omega_X}&=\frac{1}{\delta p}|f|^2_{h,\omega_X}.
    \end{align*}

    From Definition \ref{sing strictly positive of dual Nakano} 
    and Proposition \ref{dual Nakano semi-posi sing then has L2-estimate condition}, for any smooth strictly plurisubharmonic function $\psi$ on $S$, there exists $u\in L^2_{p,n-1}(S,E,he^{\delta\varphi-\psi},\omega_X)$ such that $\overline{\partial}u=f$ and 
    \begin{align*}
        \int_S|u|^2_{h,\omega_X}e^{\delta\varphi-\psi}dV_{\omega_X}\leq\int_S\langle B^{-1}_{\delta\varphi,\omega_X}f,f\rangle_{h,\omega_X}e^{\delta\varphi-\psi}dV_{\omega_X}
    \end{align*}
    if the right-hand side is finite. Taking $\psi=\delta\varphi$, we get a solution $u\in L^2_{p,n-1}(S,E,h,\omega_X)$ of $\overline{\partial}u=f$ such that 
    \begin{align*}
        \int_S|u|^2_{h,\omega_X}dV_{\omega_X}&\leq\int_S\langle B^{-1}_{\delta\varphi,\omega_X}f,f\rangle_{h,\omega_X}dV_{\omega_X}=\frac{1}{\delta p}\int_S|f|^2_{h,\omega_X}dV_{\omega_X}\\
        &\leq\frac{1}{\delta p}\int_X|f|^2_{h,\omega_X}dV_{\omega_X}<+\infty.
    \end{align*}

    Letting $u=0$ on $X\setminus S$, we have that $u\in L^2_{p,n-1}(X,E,h,\omega_X)$, $\overline{\partial}u=f$ on $X$ and 
    \begin{align*}
        \int_X|u|^2_{h,\omega_X}dV_{\omega_X}&\leq\frac{1}{\delta p}\int_X|f|^2_{h,\omega_X}dV_{\omega_X},
    \end{align*}
    from Lemma \ref{Ext d-equation for hypersurface}.
\end{proof}

\section{Applications and proof of Proposition \ref{dual Nakano semi-posi sing then has L2-estimate condition}}

In this section, as applications of Theorem \ref{smooth dual Nakano L2-condition}, 
we introduce a property necessary for proofs of Proposition \ref{dual Nakano semi-posi sing then has L2-estimate condition} and Theorem \ref{Ext Hor Thm to vect bdl} that the $(n,q)$ and $(p,n)$-$L^2$-estimate condition is preserved with respect to increasing sequences and we prove Proposition \ref{dual Nakano semi-posi sing then has L2-estimate condition}.
This phenomenon is first mentioned in \cite{Ina21a} 
as an extension of the properties seen in plurisubharmonic functions.  
After that, it is extended to the case of Nakano semi-positivity in \cite{Ina21b} and then to the case of the $(n,q)$ and $(p,n)$-$L^2$-estimate condition in \cite{Wat21}.

\begin{proposition}\label{sequence of Nakano semi-positive}$(\mathrm{cf.\,[Ina22,\,Proposition\,6.1]})$
    Let $h$ be a singular Hermitian metric on $E$.
    Assume that there exists a sequence of smooth Nakano semi-positive metrics $(h_\nu)_{\nu\in\mathbb{N}}$ increasing to $h$ pointwise.
    Then $h$ is Nakano semi-positive in the sense of Definition \ref{def Nakano semi-posi sing}, (i.e.\,$L^2$-type).
\end{proposition}

\begin{proposition}\label{sequence of dual Nakano semi-positive}$(\mathrm{cf.\,[Wat22,\,Corollary\,5.7]})$
    Let $h$ be a singular Hermitian metric on $E$.
    Assume that there exists a sequence of smooth dual Nakano semi-positive metrics $(h_\nu)_{\nu\in\mathbb{N}}$ increasing to $h$ pointwise.
    Then $h$ is $L^2$-type dual Nakano semi-positive. 
\end{proposition}

By using this proposition, we prove Proposition \ref{dual Nakano semi-posi sing then has L2-estimate condition}.

\vspace{3mm}

$\textit{Proof of Proposition \ref{dual Nakano semi-posi sing then has L2-estimate condition}}.$
Let $S$ be a Stein coordinate such that $E|_S$ is trivial. From Proposition \ref{sequence of dual Nakano semi-positive}, it is sufficient to show that 
there exists a sequence of smooth dual Nakano semi-positive metrics $(h_\nu)_{\nu\in\mathbb{N}}$ over any relatively compact subset of $S$ increasing to $h$ pointwise.

Here, $h^*$ is Nakano semi-negative singular Hermitian metric on $E^*$ over $S$. We define a sequence of smooth Hermitian metrics $(h_\nu^*)_{\nu\in\mathbb{N}}$ approximating $h^*$ by a convolution of $h^*$ with an approximate identity.
In other words, let $h^*_\nu:=h^*\ast\rho_\nu$ where $\rho_\nu$ is an approximate identity, i.e.\,$\rho\in\mathcal{D}(S)$ with $\rho\geq0,\,\rho(z)=\rho(|z|),\,\int\rho=1$ and $\rho_\nu(z)=\nu^n\rho(\nu z)$.
From Griffiths semi-negativity of $h^*$, each $h^*_\nu$ is Griffiths semi-negative and $(h_\nu^*)_{\nu\in\mathbb{N}}$ is decreasing to $h^*$ pointwise (cf.\,[BP08,\,Proposition\,3.1],\,[Rau15,\,Proposition\,1.3]).

Finally, we show that $h^*_\nu$ is Nakano semi-negative (cf.\,\cite{Rau15}). For any $n$-tuple of holomorphic sections $u=(u_1,\cdots,u_n)$ of $E$, we have locally that
\begin{align*}
    (u_j,u_k)_{h^*_\nu}(z)=\int(u_j,u_k)_{h^*_w}(z)\rho_\nu(w)dV_w,
\end{align*}
where $h^*_w(z):=h^*(z-w)$ and that 
\begin{align*}
    T^{h^*_\nu}_u(z)=\int T^{h^*_w}_u(z)\rho_\nu(w)dV_w.
\end{align*}

Since Nakano semi-negativity of $h^*_w$, for any test form $\phi\in\mathcal{D}(S)$ we have that 
\begin{align*}
    \idd T^{h^*_\nu}_u(\phi)&=\int \phi \idd T^{h^*_\nu}_u=\int T^{h^*_\nu}_u\wedge\idd\phi\\
    &=\int_z \Bigl\{\int_w T^{h^*_w}_u(z)\rho_\nu(w)dV_w\Bigr\}\wedge\idd\phi\\
    &=\int_w \Bigl\{\int_z T^{h^*_w}_u(z)\wedge\idd\phi\Bigr\}\rho_\nu(w)dV_w\\
    &=\int \idd T^{h^*_w}_u(\phi)\rho_\nu(w)dV_w\geq0,
\end{align*}
where $\idd T^{h^*_w}_u(\phi)\geq0$. Hence, $T^{h^*_\nu}_u$ is plurisubharmonic i.e. $h^*_\nu\leq_{Nak}0$ and 
we let $h_\nu:=(h^*_\nu)^*$ then $(h_\nu)_{\nu\in\mathbb{N}}$ is a sequence of smooth dual Nakano semi-positive metrics satisfying the necessary conditions.
\qed\\

Here, for convenience, we also introduce the following notion for strictly dual Nakano positivity using $L^2$-estimates.
Let $(X,\omega_X)$ be a \kah manifold, we say that $h$ is $L^2$-$\textit{type}$ $\textit{strictly dual Nakano}$ $\delta_\omega$-$\textit{positive}$  
if for any open subset $U$ and any \kah potential $\varphi$ of $\omega_X$ on $U$, $he^{\delta\varphi}$ is $L^2$-type dual Nakano semi-positive on $U$.

From Proposition \ref{dual Nakano semi-posi sing then has L2-estimate condition}, we immediately obtain the following two facts:
\begin{itemize}
    \item  Let $\omega_X$ be a \kah metric on $X$ and $h$ be a singular Hermitian metric on $E$. If $h$ is strictly dual Nakano $\delta_{\omega_X}$-positive 
    then $h$ is $L^2$-type strictly dual Nakano $\delta_{\omega_X}$-positive. 
    \item  Theorem \ref{L2-estimate for dual Nakano positive} holds under the weaker assumption that $h$ is $L^2$-type strictly dual Nakano $\delta_{\omega_X}$-positive from its proof.
\end{itemize}

\vspace{3mm}

Finally, using these two propositions, we obtain the following two theorems which is a generalization of Demailly-Skoda type theorem (see.\,\cite{DS79},\,\cite{LSY13},\,Theorem \ref{smooth Grif then Nak for tensor}). 
These theorems were shown in \cite{Ina21b} up to the Nakano (semi-)positive case, and can be shown for the dual Nakano (semi-)positive case in almost the same way using Proposition \ref{sequence of dual Nakano semi-positive}.

\begin{theorem}\label{sing Grif then Nakano for tensor}
    Let $h$ be a singular Hermitian metric on $E$. If $h$ is Griffiths semi-positive 
    then $(E\otimes\mathrm{det}\,E,h\otimes\mathrm{det}\,h)$ is Nakano semi-positive in the sense of Definition \ref{def Nakano semi-posi sing} (i.e.\,$L^2$-type) and $L^2$-type dual Nakano semi-positive.
\end{theorem}

\begin{theorem}\label{sing strictly Grif then strictly Nakano for tensor}
    Let $\omega_X$ be a \kah metric on $X$ and $h$ be a singular Hermitian metric on $E$. If $h$ is strictly Griffiths $\delta_{\omega_X}$-positive 
    then $(E\otimes\mathrm{det}\,E,h\otimes\mathrm{det}\,h)$ is strictly Nakano $(r+1)\delta_{\omega_X}$-positive in the sense of Definition \ref{sing strictly positive of Grif and Nakano} (i.e.\,$L^2$-type) and $L^2$-type strictly dual Nakano $(r+1)\delta_{\omega_X}$-positive.
\end{theorem}

\section{Proofs of Theorem \ref{Ext sing Griffiths V-thm} and \ref{sing dual Nakano V-thm}}

In this section, we get the proofs of Theorem \ref{Ext sing Griffiths V-thm} and \ref{sing dual Nakano V-thm}. 
First, we prove the following theorem and corollary, which is an extension of Theorem \ref{Hor Thm 4.4.2} to holomorphic vector bundles, to show these theorems.

\begin{theorem}\label{Ext Hor Thm to vect bdl}
    Let $X$ be a complex manifold and $E$ be a holomorphic vector bundle over $X$ equipped with a singular Hermitian metric $h$.
    We assume that $h$ is Griffiths semi-positive on $X$. Then for any $x_0\in X$, there exist an open neighborhood $U$ of $x_0$ and a \kah metric $\omega$ on $U$ satisfying that 
    for any $\overline{\partial}$-closed $f\in L^2_{p,q}(U,E\otimes\mathrm{det}\,E,h\otimes\mathrm{det}\,h,\omega)$, there exists $u\in L^2_{p,q-1}(U,E\otimes\mathrm{det}\,E,h\otimes\mathrm{det}\,h,\omega)$ such that $\overline{\partial}u=f$.
\end{theorem}

\begin{proof}
    For any $x_0 \in X$, there exist a bounded Stein neighborhood $U$ of $x_0$ such that $E|_U$ and $T_U$ are trivial and a sequence of smooth Griffiths positive metrics $(h_\nu)_{\nu\in\mathbb{N}}$ on $U$ increasing to $h$ pointwise (see [BP08, Proposition 3.1], \cite{Rau15}). 
    Here, $(E\otimes\mathrm{det}\,E,h_\nu\otimes\mathrm{det}\,h_\nu)$ is Nakano positive and $(E\otimes\mathrm{det}\,E,h\otimes\mathrm{det}\,h)$ is Nakano semi-positive in the sense of Definition \ref{def Nakano semi-posi sing} by Proposition \ref{sequence of Nakano semi-positive}. 
    We fix a bounded \kah potential $\psi$ of $\omega$ on $U$ and define the trivial vector bundle $F:=E\otimes\mathrm{det}\,E\otimes\Lambda^{n-p}T_U$ over $U$, where $\Lambda^{p,q}T^*_U\otimes E\otimes\mathrm{det}\,E\cong\Lambda^{n,q}T^*_U\otimes F$. 
    Let $I_{T_U^p}$ be a trivial Hermitian metric and $h_{T_U^p}:=I_{T_U^p}e^{-\psi}$ be a smooth Nakano positive metric on a trivial vector bundle $\Lambda^{n-p}T_U$. 
    
    Define a singular Hermitian metric $h_F:=h\otimes\mathrm{det}\,h\otimes h_{T^p_U}$ on trivial vector bundle $F$ over $U$. Then we have that $(F,h_F)$ is strictly Nakano $1_\omega$-positive on $U$ in the sense of Definition \ref{sing strictly positive of Grif and Nakano}. 
    This is enough to show that for any \kah potential $\varphi$ of $\omega$, $h_Fe^\varphi$ is Nakano semi-positive on $U$ in the sense of Definition \ref{def Nakano semi-posi sing}. Let $h_{F,\nu}:=h_\nu\otimes\mathrm{det}\,h_\nu\otimes h_{T^p_U}$ be a smooth Hermitian metric on $F$ then $h_{F,\nu}e^\varphi$ is Nakano semi-positive.
    In fact, from $h_{F,\nu}e^\varphi=h_\nu\otimes\mathrm{det}\,h_\nu\otimes I_{T_U^p}e^{-\psi+\varphi}$ and Nakano positivity of $h_\nu\otimes\mathrm{det}\,h_\nu$, we get 
    \begin{align*}
        i\Theta_{F,h_{F,\nu}e^\varphi}&=i\Theta_{E\otimes\mathrm{det}\,E,h_\nu\otimes\mathrm{det}\,h_\nu}\otimes\mathrm{id}_{\Lambda^{n-p}T_U}+\idd(-\psi+\varphi)\otimes\mathrm{id}_F\\
        &=i\Theta_{E\otimes\mathrm{det}\,E,h_\nu\otimes\mathrm{det}\,h_\nu}\otimes\mathrm{id}_{\Lambda^{n-p}T_U}\geq_{Nak}0.
    \end{align*}
    Therefore $(h_{F,\nu}e^\varphi)_{\nu\in\mathbb{N}}$ is a sequence of smooth Nakano semi-positive metric on $U$ increasing to $h_Fe^\varphi$. From Proposition \ref{sequence of Nakano semi-positive}, we have that $h_Fe^\varphi$ is Nakano semi-positive on $U$ in the sense of Definition \ref{def Nakano semi-posi sing}.

    For any $(p,q)$-form $v$ with values in $E\otimes\mathrm{det}\,E$, the form $v$ is considered $(n,q)$-form with values in $F$ and we have that $|v|^2_{h_F,\omega}=|v|^2_{h\otimes\mathrm{det}\,h,\omega}e^{-\psi}$. In fact, we can write
    \begin{align*}
        \omega&=\sum dz_j\wedge d\overline{z}_j, \\
        v&=\sum v_{IJ\lambda}dz_I\wedge d\overline{z}_J\otimes e_\lambda\in \Lambda^{p,q}T^*_U\otimes E\otimes\mathrm{det}\,E
    \end{align*}
    \begin{align*}
        &=\sum v_{IJ\lambda}dz_N\wedge d\overline{z}_J\otimes \frac{\partial}{\partial z_{N\setminus I}} \otimes e_\lambda\in \Lambda^{n,q}T^*_U\otimes \Lambda^{n-p}T_U\otimes E\otimes\mathrm{det}\,E\\
        &\qquad\qquad\qquad\qquad\qquad\qquad\qquad\quad\,\cong \Lambda^{n,q}T^*_U\otimes F,
    \end{align*}
    at any fixed point, where $(e_j)_{1\leq j\leq r}$ is an orthonormal basis of $E\otimes\mathrm{det}\,E$. Then we get 
    \begin{align*}
        |v|^2_{h_F,\omega}&=|v|^2_{h\otimes\mathrm{det}\,h\otimes I_{T^p_U},\omega}e^{-\psi}=\sum v_{IJ\lambda}\overline{v}_{KJ\mu}(h\otimes\mathrm{det}\,h)_{\lambda\mu}\delta_{N\setminus I,N\setminus K}e^{-\psi}\\
        &=\sum v_{IJ\lambda}\overline{v}_{IJ\mu}(h\otimes\mathrm{det}\,h)_{\lambda\mu}e^{-\psi}=|v|^2_{h\otimes\mathrm{det}\,h,\omega}e^{-\psi},
    \end{align*}
    where $\delta_{IK}=\Pi_j \delta_{i_jk_j}$ is multi-Kronecker's delta. 

    By using the boundedness of $\psi$ on $U$, for any $\overline{\partial}$-closed $f\in L^2_{p,q}(U,E\otimes\mathrm{det}\,E,h\otimes\mathrm{det}\,h,\omega)$, we have that 
    \begin{align*}
        \int_U|f|^2_{h_F,\omega}dV_\omega=\int_U|f|^2_{h\otimes\mathrm{det}\,h,\omega}e^{-\psi}dV_\omega\leq\sup_Ue^{-\psi}\int_U|f|^2_{h\otimes\mathrm{det}\,h,\omega}dV_\omega<+\infty,
    \end{align*}
    i.e. $f\in L^2_{n,q}(U,F,h_F,\omega)$, where $\sup_Ue^{-\psi}<+\infty$. Therefore, from strictly Nakano $1_\omega$-positivity of $(F,h_F)$ in the sense of Definition \ref{sing strictly positive of Grif and Nakano} and Theorem \ref{L2-estimate for Nakano positive}, there exists $u\in L^2_{n,q-1}(U,F,h_F,\omega)$ such that $\overline{\partial}u=f$ and 
    \begin{align*}
        \int_U|u|^2_{h_F,\omega}dV_\omega\leq\frac{1}{q}\int_U|f|^2_{h_F,\omega}dV_\omega
        <+\infty.
    \end{align*}
    Repeating the above argument, we have that 
    \begin{align*}
        +\infty>\int_U|u|^2_{h_F,\omega}dV_\omega=\int_U|u|^2_{h\otimes\mathrm{det}\,h,\omega}e^{-\psi}dV_\omega\geq\inf_Ue^{-\psi}\int_U|u|^2_{h\otimes\mathrm{det}\,h,\omega}dV_\omega,
    \end{align*}
    i.e. $u\in L^2_{p,q-1}(U,E\otimes\mathrm{det}\,E,h\otimes\mathrm{det}\,h,\omega)$ and the following $L^2$-estimate
    \begin{align*}
        \int_U|u|^2_{h\otimes\mathrm{det}\,h,\omega}dV_\omega\leq\frac{1}{q}\frac{\sup_Ue^{-\psi}}{\inf_Ue^{-\psi}}\int_U|f|^2_{h\otimes\mathrm{det}\,h,\omega}dV_\omega,
    \end{align*}
    where the right-hand side is finite.
\end{proof}

\begin{corollary}\label{Cor ext Hor thm to vect bdl}
    Let $X$ be a complex manifold and $E$ be a holomorphic vector bundle over $X$ equipped with a singular Hermitian metric $h$.
    We assume that $h$ is Griffiths semi-positive on $X$ and that $\mathrm{det}\,h$ is bounded on $X$. 
    Then for any $x_0\in X$, there exist an open neighborhood $U$ of $x_0$ and a \kah metric $\omega$ on $U$ satisfying that 
    for any $\overline{\partial}$-closed $f\in L^2_{p,q}(U,E,h,\omega)$, there exists $u\in L^2_{p,q-1}(U,E,h,\omega)$ such that $\overline{\partial}u=f$.
\end{corollary}

\begin{proof}
    For any $x_0$, there exists a bounded Stein neighborhood $U$ of $x_0$ such that $E|_U$ is trivial then $E\otimes\mathrm{det}\,E|_U\cong E|_U$. By the boundedness of $\mathrm{det}\,h$ on $U$, 
    for any $\overline{\partial}$-closed $f\in L^2_{p,q}(U,E,h,\omega)$ we have that 
    \begin{align*}
        \int_U|f|^2_{h\otimes\mathrm{det}\,h,\omega}dV_\omega=\int_U|f|^2_{h,\omega}\mathrm{det}\,h\,dV_\omega\leq\sup_U\mathrm{det}\,h\int_U|f|^2_{h,\omega}dV_\omega<+\infty,
    \end{align*}
    i.e. $f\in L^2_{p,q}(U,E\otimes\mathrm{det}\,E,h\otimes\mathrm{det}\,h,\omega)$.

    From Theorem \ref{Ext Hor Thm to vect bdl}, there exists $u\in L^2_{p,q-1}(U,E\otimes\mathrm{det}\,E,h\otimes\mathrm{det}\,h,\omega)$ satisfies $\overline{\partial}u=f$.
    By the Griffiths semi-positivity of $h$, we have that $\inf_U\mathrm{det}\,h>0$ and that 
    \begin{align*}
        \inf_U\mathrm{det}\,h\int_U|u|^2_{h,\omega}dV_\omega\leq\int_U|u|^2_{h,\omega}\mathrm{det}\,h\,dV_\omega=\int_U|u|^2_{h\otimes\mathrm{det}\,h,\omega}dV_\omega<+\infty,
    \end{align*}
    i.e. $u\in L^2_{p,q-1}(U,E,h,\omega)$.
\end{proof}

Finally, we prove Theorem \ref{Ext sing Griffiths V-thm} and \ref{sing dual Nakano V-thm} by using the above result.
For singular Hermitian metrics $h$ on $E$, we define the subsheaf $\mathscr{L}^{p,q}_{E,h}$ of germs of $(p,q)$-forms $u$ with values in $E$ and with measurable coefficients such that both $|u|^2_{h}$ and $|\overline{\partial}u|^2_h$ are locally integrable.

\vspace{3mm}

$\textit{Proof of Theorem \ref{Ext sing Griffiths V-thm}}.$
    We already know from Griffiths vanishing theorem with singular Hermitian metric about the case of $(n,q)$-forms, i.e. $H^q(X,K_X\otimes \mathscr{E}(h\otimes\mathrm{det}\,h))=0$ for $q>0$ (see [Ina22,\,Theorem\,1.6]). 

    We consider the following sheaves sequence:
    \begin{align*}
        \xymatrix{
        0 \ar[r] & \mathrm{ker}\,\overline{\partial}_0 \hookrightarrow \mathscr{L}^{p,0}_{E\otimes\mathrm{det}\,E,h\otimes\mathrm{det}\,h} \ar[r]^-{\overline{\partial}_0} & \mathscr{L}^{p,1}_{E\otimes\mathrm{det}\,E,h\otimes\mathrm{det}\,h} \ar[r]^-{\overline{\partial}_1} & \cdots \ar[r]^-{\overline{\partial}_{n-1}} & \mathscr{L}^{p,n}_{E\otimes\mathrm{det}\,E,h\otimes\mathrm{det}\,h} \ar[r] & 0.
    }
    \end{align*}
    Since Theorem \ref{Ext Hor Thm to vect bdl}, we have that the above sheaves sequence is exact.

    From Lemma \ref{Regularity of currents for (p,0)-forms}, the kernel of $\overline{\partial}_0$ consists of all germs of holomorphic $(p,0)$-forms with values in $E\otimes\mathrm{det}\,E$ which satisfy the integrability condition
    and we have that $\mathrm{ker}\,\overline{\partial}_0=\Omega_X^p\otimes \mathscr{E}(h\otimes\mathrm{det}\,h)$.
    In fact, for any locally open subset $U\subset \mathbb{C}^n$ we obtain
    \begin{align*}
        f\in\mathrm{ker}\,\overline{\partial}_0(U) &\iff f=\sum f_Idz_I \in H^0(U,\Omega^p_X\otimes E\otimes\mathrm{det}\,E)~\,\, \mathrm{such~ that}  \\
        & \int_U|f|^2_{h\otimes\mathrm{det}\,h,\omega}dV_\omega=\sum_I\int_U|f_I|^2_{h\otimes\mathrm{det}\,h}dV_\omega<+\infty.
    \end{align*}
    Therefore any $f_I\in H^0(U,E\otimes\mathrm{det}\,E)$ satisfy the condition $f_I\in\mathscr{E}(h\otimes\mathrm{det}\,h)(U)$.

    Since acyclicity of each $\mathscr{L}^{p,q}_{E\otimes\mathrm{det}\,E,h\otimes\mathrm{det}\,h}$, we have that 
    \begin{align*}
        H^q(X,\Omega_X^p\otimes \mathscr{E}(h\otimes\mathrm{det}\,h))\cong H^q(\Gamma(X,\mathscr{L}^{p,\bullet}_{E\otimes\mathrm{det}\,E,h\otimes\mathrm{det}\,h})).
    \end{align*} 
    By Theorem \ref{sing strictly Grif then strictly Nakano for tensor} and \ref{L2-estimate for dual Nakano positive}, we get $H^n(\Gamma(X,\mathscr{L}^{p,\bullet}_{E\otimes\mathrm{det}\,E,h\otimes\mathrm{det}\,h}))=0$.    \qed\\

$\textit{Proof of Theorem \ref{sing dual Nakano V-thm}}.$
    We consider the following sheaves sequence:
    \begin{align*}
        \xymatrix{
        0 \ar[r] & \mathrm{ker}\,\overline{\partial}_0 \hookrightarrow \mathscr{L}^{p,0}_{E,h} \ar[r]^-{\overline{\partial}_0} & \mathscr{L}^{p,1}_{E,h} \ar[r]^-{\overline{\partial}_1} & \cdots \ar[r]^-{\overline{\partial}_{n-1}} & \mathscr{L}^{p,n}_{E,h} \ar[r] & 0.
    }
    \end{align*}
    Since Corollary \ref{Cor ext Hor thm to vect bdl}, we have that the above sheaves sequence is exact.

    Locally, we see $h=\mathrm{det}\,h\cdot \widehat{h^*}$ where $\widehat{h^*}$ is the adjugate matrix of $h^*$. Since Griffiths semi-negativity of $h^*$, each element of $\widehat{h^*}$ is locally bounded [PT18, Lemma 2.2.4].
    From the assumption $\mathrm{det}\,h$ is bounded, we get $\mathscr{E}(h)=\mathscr{O}(E)$.

    Repeating the above argument, we have that $\mathrm{ker}\,\overline{\partial}_0=\Omega_X^p\otimes \mathscr{E}(h)=\Omega_X^p\otimes \mathscr{O}(E)$.
    Since acyclicity of each $\mathscr{L}^{p,q}_{E,h}$, we have that $H^q(X,\Omega_X^p\otimes E)\cong H^q(\Gamma(X,\mathscr{L}^{p,\bullet}_{E,h}))$.
    By Theorem \ref{L2-estimate for dual Nakano positive}, we have that $H^n(\Gamma(X,\mathscr{L}^{p,\bullet}_{E,h}))=0$.    \qed\\

\begin{remark}
    Since Theorem \ref{L2-estimate for dual Nakano positive} holds, Theorem \ref{sing dual Nakano V-thm} also holds under the weaker assumption that $h$ is $L^2$-type strictly dual Nakano $\delta_{\omega_X}$-positive.
\end{remark}

    \vspace{3mm}

{\bf Acknowledgement. } 
I would like to thank my supervisor Professor Shigeharu Takayama for guidance and helpful advice. 
I would also like to thank Professor Takahiro Inayama for useful advice on the strength relationship of the definitions.

$ $

\rightline{\begin{tabular}{c}
    $\it{Yuta~Watanabe}$ \\
    $\it{Graduate~School~of~Mathematical~Sciences}$ \\
    $\it{The~University~of~Tokyo}$ \\
    $3$-$8$-$1$ $\it{Komaba, ~Meguro}$-$\it{ku}$ \\
    $\it{Tokyo, ~Japan}$ \\
    ($E$-$mail$ $address$: watayu@g.ecc.u-tokyo.ac.jp)
\end{tabular}}

\end{document}